\documentclass[11pt]{amsart}
\numberwithin{equation}{section}
\usepackage{hyperref}
\hypersetup{nesting=true,debug=true,naturalnames=true}
\usepackage{graphicx,amssymb,upref}
\usepackage[usenames,dvipsnames]{pstricks}
\usepackage{epsfig}
\usepackage{color}
\usepackage[utf8]{inputenc}
\usepackage[T1]{fontenc}
\usepackage{amsmath, amsthm}
\usepackage[english]{babel}
\usepackage{amssymb}
\usepackage{latexsym}
\usepackage{graphicx,amssymb,upref}
\usepackage[all]{xy}
\usepackage{tikz}
\usetikzlibrary{matrix} 
\usepackage{mathtools}
\usepackage{pdfpages}
\usepackage{multirow}
\usepackage{comment}
\allowdisplaybreaks

\date{\today}

\newtheorem{theorem}{Theorem}[section]
\newtheorem{lemma}[theorem]{Lemma}
\newtheorem{proposition}[theorem]{Proposition}
\newtheorem{corollary}[theorem]{Corollary}
\theoremstyle{definition}
\newtheorem{definition}[theorem]{Definition}
\newtheorem{example}[theorem]{Example}

\newtheorem{remark}[theorem]{Remark}

\DeclareMathOperator{\g}{\mathfrak{g}}

\begin{document}

\title[Affinization: Leibniz algebras]{Affinization of algebraic structures: Leibniz algebras}

\author[Brzezi\'nski]{Tomasz Brzezi\'nski}

\address{
Department of Mathematics, Swansea University,
Fabian Way,
  Swansea SA1 8EN, U.K.\ \newline \indent
Faculty of Mathematics, University of Bia{\l}ystok, K.\ Cio{\l}kowskiego  1M,
15-245 Bia\-{\l}ys\-tok, Poland}

\email{T.Brzezinski@swansea.ac.uk}

\author[Radziszewski]{Krzysztof Radziszewski}

\address{Doctoral School \&
Faculty of Mathematics, University of Bia{\l}ystok, K.\ Cio{\l}kowskiego,
15-245 Bia\-{\l}ys\-tok, Poland}

\email{K.Radziszewski@uwb.edu.pl}

\author[Ramos Pérez]{Brais Ramos Pérez}

\address{CITMAga, 15782 Santiago de Compostela, Spain.\newline \indent
 Faculty of Mathematics, Department of Mathematics, University of Santiago de Compostela, 15782 Santiago de Compostela, Spain.}

 \email{braisramos.perez@usc.es}

\begin{abstract}
A general procedure of affinization of linear algebra structures is illustrated by the case of Leibniz algebras. Specifically, the definition of an affine Leibniz bracket, that is, a bi-affine operation on an affine space that at each tangent vector space becomes a (bi-linear) Leibniz bracket in terms of a tri-affine operation called a Leibnizian, is given. An affine space together with such an operation is called a Leibniz affgebra. It is shown that any Leibniz algebra can be extended to a family of Leibniz affgebras. Depending on the choice of a Leibnizian different types of Leibniz affgebras are introduced. These include: derivative-type, which captures the derivation property of linear Leibniz bracket;  homogeneous-type, which is based on the simplest and least restrictive choice of the Leibnizian; Lie-type which includes all Lie affgebras introduced in [R.R.\ Andruszkiewicz, T.\ Brzezi\'nski \& K.\ Radziszewski, {\it Lie affgebras vis-à-vis Lie algebras}, Res.\ Math.\  {\bf 80} (2025), art.\ 61.]. Each type is illustrated by examples with prescribed Leibniz algebra fibres.
\end{abstract} 

\vspace{0.2cm}



\maketitle
\section{Introduction}\label{sec.intro}
Informally and in broad terms, by affinization we mean a process which converts an algebraic structure with a nullary operation (i.e., in which a chosen element plays special role) into one which has no nullary operations, but is such that by a free choice of any element it is retracted to the original structure. The prime example of this procedure is the conversion of groups into heaps, by using the ternary operation $(a,b,c)\mapsto ab^{-1}c$ \cite{Pru:the}, \cite{Bae:ein}. Once the affinization is achieved it is the affinized structure which becomes the fundamental object from which the original structure can be derived. As an illustration of this point of view, one can trace the process by which a choice of a coordinate system (or, indeed, a point which serves as the origin for this system) converts  an affine Euclidean space into a vector space. The appearance of the latter can be seen as a matter of convenience with an aim of achieving an ease of description. One can argue that the affine point of view rather than a linear one is dictated by the nature. The geometry of the space is independent of the coordinate system. In mechanics, the motion is independent of the choice of the frame of reference or the observer.  In high energy physics, fundamental interactions are independent of the gauge. It is only the mathematical description of geometry, motion or fundamental interactions which depends on the choice of coordinate system, the position of the observer or the gauge. The search for frame independent description of mechanics led W. Tulczyjew and others   to formulate a programme in which geometric or mechanical phenomena that are traditionally described in terms of vector spaces are described in terms of affine spaces; see e.g.\ \cite{Tul:fra}, \cite{Tul:ori}, \cite{BenTul:can}, \cite{GraGra:av1}, \cite{GraGra:fra}, \cite{GraGra:av2}.

Apart from its conceptual significance, the affinization leads to unification of algebraic structures. From the point of view of the heap interpretation of groups, both subgroups and their cosets are sub-heaps \cite{Certaine}. More recently, an attempt to understand braces, an algebraic structure introduced in the context of the set-theoretic Yang-Baxter equation \cite{Rum:bra}, \cite{CedJes:bra}, has led to the unification of braces with rings in terms of their mutual affinization known as trusses \cite{Brz:tru}, \cite{BRZ1}. The idea here is to replace the additive group structure by a corresponding heap and demand distributivity of multiplication over the heap operation.

Let us summarize affinization of bilinear structures briefly. The process of affinization starts with a vector-free formulation of the notion of an affine space. This is quite well known and we recall it in Section~\ref{sec.pre} following \cite{BreBrz:hea}. In this formulation, at each point $o$ of the affine space $\mathfrak{a}$, there is a unique vector space with $o$ as the zero vector, which we call a {\em vector fibre} or the {\em tangent space} at $o$ and denote it by $T_o\mathfrak{a}$. As a set, or indeed as an affine space, $T_o\mathfrak{a}$ is isomorphic to $\mathfrak{a}$. Furthermore, vector fibres of $\mathfrak{a}$ at different points of $\mathfrak{a}$ are mutually isomorphic (as vector spaces). Suppose that we have an algebraic structure with a bilinear multiplication $\bullet$ satisfying certain axioms (such as the associative law, Jacobi identity and antisymmetry, or the self-derivation property, which is discussed in the current text). We aim at achieving a bi-affine operation $*$  on an affine space $\mathfrak{a}$ such that for any $o\in \mathfrak{a}$, the uniquely linearized version of $*$ on $T_o\mathfrak{a}$ satisfies the axioms of $\bullet$. This means that, written in $T_o\mathfrak{a}$, $*$ necessarily has the form
\begin{equation}\label{bullet.star}
    a * b = a\bullet b +\lambda(a) +\mu(b) + s,
\end{equation}
for all $a,b\in \mathfrak{a}$, where $\lambda, \mu$ are uniquely derived linear operators on $T_o\mathfrak{a}$, $s\in \mathfrak{a}$ is also uniquely derived, and $\bullet$ satisfies the specified axioms. In this way we arrive at three separate problems:
\begin{description}
    \item[Q1] What axioms need to be imposed on the bi-affine operation $*$  on $\mathfrak{a}$ to guarantee the axioms for $\bullet$ on each fibre of $\mathfrak{a}$?
    \item[Q2] What necessary and sufficient conditions the data $\lambda$, $\mu$ and $s$ in equation \eqref{bullet.star} need to satisfy to fulfill the axioms imposed in answer to \textbf{Q1}?
    \item[Q3] Can every example of a specified linear structure give rise to the corresponding example of the affine structure?
\end{description}

To indicate the role of affine spaces and borrowing the neologism coined in \cite{GGU} we refer to an affinized version of an algebra as to an {\em affgebra}.
       
This programme has been carried out successfully for associative and Lie algebras. In the former case, the associative law need not be modified in any way. The requirement of bi-affine rather than  bilinear condition on multiplication leads to trusses (this answers \textbf{Q1} above), while $\lambda$, $\mu$ and $s$ form a (linearized) {\em homothetic datum} \cite{AndBrzRyb:ext}, an enriched version of {\em double homothetism} \cite{Red:ver} or {\em self-permutable bimultiplication} \cite{Mac:ext} introduced in the context of ring extensions, which answers \textbf{Q2}. The first attempt to affinize Lie algebras was made in \cite{GGU}, but such definition was dependent on a specified vector space. The vector-space independent formulation was given in \cite{BRZ2} and improved upon in \cite{BRZ3}. In this case both the antisymmetry and Jacobi rules need to be modified. This is the consequence of the fact that there is no binary operation on an affine space, but the ternary one (translating the first point by the vector from the second to the third one). As a consequence, equalities are possible only between odd combinations of (non-zero, since the operation is no nullary) terms. In order to answer \textbf{Q2} in this setting, no conditions on $s$ are imposed, while $\lambda$ and $\mu$ form a generalized derivation in the sense of \cite{LL} of a specific kind on the fibre Lie algebra. The answer to question {\bf Q3} is affirmative for associative and Lie affgebras.

The aim of this paper is to present general conditions for an affinization of a Leibniz algebra and then discuss three special cases: \emph{derivative}, \emph{homogeneous} and \emph{Lie-type} Leibniz affgebras, each offering distinct advantages and disadvantages to be discussed throughout the article. 
\section{Preliminaries}\label{sec.pre}
Along the paper, we will denote by $\mathbb{F}$ a field, a priori without any restriction about its characteristic.

\subsection{Leibniz algebras}
\begin{definition}\label{def.Leibniz}
A \textbf{Leibniz algebra} is a pair $(\mathfrak{l},[-,-])$, where $\mathfrak{l}$ is an $\mathbb{F}$-vector space and $[-,-]\colon \mathfrak{l}\times \mathfrak{l}\rightarrow \mathfrak{l}$ is a bilinear map, such that the following condition holds:
\begin{equation}\label{LeibRule}
[[a,b],c]=[[a,c],b]+[a,[b,c]]\quad\textnormal{(Leibniz rule)}
\end{equation}
for all $a,b,c\in\mathfrak{l}$. The map $[-,-]$ is referred to as the \textbf{Leibniz bracket} of $\mathfrak{l}$.

Let $(\mathfrak{l},[-,-]_{\mathfrak{l}})$ and $(\mathfrak{m},[-,-]_{\mathfrak{m}})$ be Leibniz algebras. A linear map $f\colon \mathfrak{l}\rightarrow \mathfrak{m}$ is a \textbf{Leibniz algebra morphism} if $f$ preserves the bracket, i.e., $f([a,b]_{\mathfrak{l}})=[f(a),f(b)]_{\mathfrak{m}}$ for all $a,b\in\mathfrak{l}$. The category of Leibniz algebras will be denoted by {\sf Leib}.
\end{definition}
\begin{remark}\label{Lie_algebra_is_Leibniz_algebra}
It is straightforward to prove that every Lie algebra $(\g,[-,-])$ satisfies the Leibniz rule \eqref{LeibRule}. This yields an inclusion functor $\iota\colon {\sf Lie}\longrightarrow{\sf Leib},$ where ${\sf Lie}$ denotes the category of Lie algebras.

Note that every Leibniz algebra $\mathfrak{l}$ always has the following non-trivial ideal
\[\textnormal{Leib}(\mathfrak{l})\coloneqq \operatorname{span}_{\mathbb{F}}\{[x,x]\colon x\in\mathfrak{l}\}.\]
The quotient $\mathfrak{l}/\textnormal{Leib}(\mathfrak{l})$ is always a Lie algebra, so $\mathfrak{l}$ is a Lie algebra if and only if $\textnormal{Leib}(\mathfrak{l})=0$. In other words, $\textnormal{Leib}(\mathfrak{l})$ allows to measure how far a Leibniz algebra is from being a Lie algebra. 
\end{remark}
\begin{definition} 
Let $(\mathfrak{l},[-,-])$ be a Leibniz algebra. 
\begin{enumerate}
\item[(i)] The \textbf{left center} of $\mathfrak{l}$ is the subset of $\mathfrak{l}$
\[\mathcal{LZ}(\mathfrak{l})\coloneqq\{x\in\mathfrak{l}\colon [x,\mathfrak{l}]=0\}.\]
\item[(ii)] The \textbf{right center} of $\mathfrak{l}$ is the subset of $\mathfrak{l}$
\[\mathcal{RZ}(\mathfrak{l})\coloneqq\{x\in\mathfrak{l}\colon [\mathfrak{l},x]=0\}.\]
\end{enumerate}
\end{definition}
\begin{remark}\label{rem.condLeibR}
In every Leibniz algebra $(\mathfrak{l},[-,-])$, \begin{equation}\label{condLeibR}[a,[b,b]]=0,\end{equation} 
for all $a,b\in\mathfrak{l},$ i.e., $\textnormal{Leib}(\mathfrak{l})\subseteq\mathcal{RZ}(\mathfrak{l})$. Indeed, by the Leibniz rule \eqref{LeibRule}, 
\[[[a,b],b]-[[a,b],b]-[a,[b,b]]=0\implies[a,[b,b]]=0.\]
\end{remark}
\begin{definition}
    A Leibniz algebra $(\mathfrak{l},[-,-])$ is \textbf{abelian} if $[\mathfrak{l},\mathfrak{l}]=0$.
\end{definition}
As a consequence of \eqref{condLeibR}, if $(\mathfrak{l},[-,-])$ is a Leibniz algebra with idempotent bracket, i.e., $[a,a]=a$ for all $a\in\mathfrak{l}$, then $\mathfrak{l}$ is automatically abelian.

The notion of derivation of a Leibniz algebra is defined analogously to that in general algebraic structures.
\begin{definition}
    Let $(\mathfrak{l},[-,-])$ be a Leibniz algebra. A \textbf{derivation} of $\mathfrak{l}$ is a linear endomorphism $L\colon \mathfrak{l}\rightarrow \mathfrak{l}$ satisfying the Leibniz rule with respect to the Leibniz bracket:
\begin{equation*}\label{der_def}
[L(a),b]+[a,L(b)]=L([a,b])
\end{equation*}
for all $a,b\in\mathfrak{l}$. We will denote the collection of all derivations of $\mathfrak{l}$ by $\operatorname{Der}(\mathfrak{l})$.
\end{definition}

Thus, equation \eqref{LeibRule} is equivalent to the fact that $\operatorname{ad}_{a}\coloneqq[-,a]\in\operatorname{Der}(\mathfrak{l})$ for all $a\in\mathfrak{l}$.

\subsection{Heaps and affine spaces}
\begin{definition}
A \textbf{heap} is a set $X$ together with a ternary operation $\langle-,-,-\rangle\colon X\times X\times X\rightarrow X$ which satisfies the following conditions:
\begin{subequations}\label{heapcond}
\begin{align}\label{h_asoc}
\langle\langle a,b,c\rangle,d,e\rangle=\langle a,b,\langle c,d,e\rangle\rangle\quad&\textnormal{(associativity),}\\
\label{h_Malcev1}\langle a,b,b\rangle=a \quad&\textnormal{(Mal'cev identity 1),}\\
\label{h_Malcev2}\langle a,a,b\rangle=b \quad &\textnormal{(Mal'cev identity 2),}
\end{align}
\end{subequations}
for all $a,b,c,d,e\in X$. A heap $(X,\langle -,-,-\rangle)$ is said to be \textbf{abelian} if the equality
\begin{equation}\label{h_ab}
\langle a,b,c\rangle=\langle c,b,a\rangle
\end{equation}
also holds. 

Let $(X,\langle -,-,-\rangle_{X})$ and $(Y,\langle -,-,-\rangle_{Y})$ be heaps. A map $f\colon X\rightarrow Y$ is a \textbf{heap morphism} if $f$ preserves the heap ternary operation, i.e., $f(\langle a,b,c\rangle_{X})=\langle f(a),f(b),f(c)\rangle_{Y}$ for all $a,b,c\in X$.
\end{definition}

Due to \eqref{h_asoc}, we simply write $\langle a,b,c,d,e\rangle$ for $\langle\langle a,b,c\rangle,d,e\rangle=\langle a,b,\langle c,d,e\rangle\rangle$ because the distribution of brackets does not matter. Note that if $X$ is an abelian heap, then \eqref{h_Malcev1} is equivalent to \eqref{h_Malcev2}.

There exists a strong relationship between abelian heaps and abelian groups. On the one hand, if $(G,+)$ is an abelian group, then $G$ with the ternary operation defined by
\[\langle x,y,z\rangle_{G}\coloneqq x-y+z,\]
for all $x,y,z\in G$, is an abelian heap. On the other hand, if $(X,\langle -,-,-\rangle)$ is an abelian heap and $o\in X$, then the binary operation
\[a+b\coloneqq\langle a,o,b\rangle,\]
for all $a,b\in X$, induces an abelian group structure on $X$. The unit is the fixed element $o$, and the inverse of $a\in X$ with respect to $+$ is $-a\coloneqq \langle o,a,o\rangle$. This group structure is called \textbf{the retract of $X$ at $o$}, denoted by $X_{o}$. All retracts of a given heap are mutually isomorphic (as groups) (see \cite[Lemma 2.1]{BRZ1}).

\begin{definition}\label{affspace}
An \textbf{$\mathbb{F}$-affine space} (or simply, \textbf{affine space}) is a triple $(\mathfrak{a},\langle -,-,-\rangle, -\triangleright_{-}-)$, where $(\mathfrak{a},\langle -,-,-\rangle)$ is a non-empty abelian heap and $-\triangleright_{-}-$ is a ternary map:
\begin{align*}
-\triangleright_{-}-\colon \mathbb{F}\times\mathfrak{a}\times\mathfrak{a}&\rightarrow\mathfrak{a}\\(\alpha,a,b)&\mapsto \alpha\triangleright_{a}b,
\end{align*}
satisfying the following conditions:
\begin{enumerate}
\item[(i)] For all $\alpha\in\mathbb{F}$ and $a\in \mathfrak{a}$, the map $\alpha\triangleright_{a}-\colon\mathfrak{a}\rightarrow \mathfrak{a}$ is a heap morphism.
\item[(ii)] For all $a,b\in\mathfrak{a}$, the map $-\triangleright_{a}b\colon\mathbb{F}\rightarrow\mathfrak{a}$ is a heap morphism considering that the heap ternary operation in $\mathbb{F}$ is given by $\langle\alpha,\beta,\gamma\rangle_{\mathbb{F}}\coloneqq \alpha-\beta+\gamma$ for all $\alpha,\beta,\gamma\in\mathbb{F}$.
\item[(iii)] For all $a\in\mathfrak{a}$, the map $-\triangleright_{a}-\colon\mathbb{F}\times\mathfrak{a}\rightarrow \mathfrak{a}$ is a left action, i.e., $1_{\mathbb{F}}\triangleright_{a}b=b$ and $\alpha\triangleright_{a}(\beta\triangleright_{a}b)=(\alpha\beta)\triangleright_{a}b$, for all $\alpha,\beta\in\mathbb{F}$ and for all $b\in\mathfrak{a}$.
\item[(iv)] For all $a,b\in\mathfrak{a}$, $0_{\mathbb{F}}\triangleright_{a}b=a$.
\item[(v)] (\textbf{Base change property})
\[\alpha\triangleright_{a}b=\langle \alpha\triangleright_{c}b,\alpha\triangleright_{c}a,a\rangle\]
for all $\alpha\in\mathbb{F}$ and for all $a,b,c\in \mathfrak{a}$.
\end{enumerate}

Let $(\mathfrak{a},\langle -,-,-\rangle_{\mathfrak{a}}, -\triangleright^{\mathfrak{a}}_{-}-)$ and $(\mathfrak{b},\langle -,-,-\rangle_{\mathfrak{b}}, -\triangleright^{\mathfrak{b}}_{-}-)$ be affine spaces. A map $f\colon \mathfrak{a}\rightarrow\mathfrak{b}$ is an \textbf{affine morphism} if $f$ is a heap morphism and also the equality
\begin{equation}\label{affmap_cond}
f(\alpha\triangleright^{\mathfrak{a}}_{a}b)=\alpha\triangleright^{\mathfrak{b}}_{f(a)}f(b)
\end{equation}
holds, for all $a,b\in\mathfrak{a}$ and for all $\alpha\in\mathbb{F}$.
\end{definition}

The following two properties are deduced from the axioms of an affine space (see \cite[Lemma 3.5]{BreBrz:hea} for detailed proofs):
\begin{enumerate}
\item[(i)] The map $\alpha\triangleright_{-}b\colon \mathfrak{a}\rightarrow\mathfrak{a}$ is a heap morphism for all $\alpha\in\mathbb{F}$ and $b\in\mathfrak{a}$.
\item[(ii)] For all $\alpha\in \mathbb{F}$ and $a\in\mathfrak{a}$, the equality 
\begin{equation}\label{prop_affspace}\alpha\triangleright_{a}a=a\end{equation}
holds.
\end{enumerate}
\begin{remark}\label{morph_aff_lin}
Let $(\mathfrak{a},\langle -,-,-\rangle, -\triangleright_{-}-)$ be an affine space and fix $o\in\mathfrak{a}$. The retract of $\mathfrak{a}$ at $o$, $\mathfrak{a}_{o}$, is an $\mathbb{F}$-vector space with the external operation
\[\alpha a\coloneqq\alpha\triangleright_{o}a,\]
for all $\alpha\in\mathbb{F}$ and for all $a\in\mathfrak{a}.$ This vector space is called \textbf{the tangent space to $\mathfrak{a}$ at $o$} (or simply, the \textbf{fibre of $\mathfrak{a}$ at $o$}), and  denoted by $T_{o}\mathfrak{a}$.
In terms of the external operation in $T_{o}\mathfrak{a}$,
\begin{equation}\label{relation_external_triangle}
\alpha\triangleright_{a}b=(1-\alpha)a+\alpha b.
\end{equation}
Moreover, $T_{o}\mathfrak{a}$ acts freely and transitively over $\mathfrak{a}$ via the following action
\begin{align*}
\psi\colon \mathfrak{a}\times T_{o}\mathfrak{a}\rightarrow\mathfrak{a}\qquad (a,b)\mapsto \psi(a,b)\coloneqq \langle a,o,b\rangle.
\end{align*}
Note also that fibres of $\mathfrak{a}$ are mutually isomorphic (as vector spaces) (see \cite[Remark 2.5]{BRZ2}).

In addition, if $f\colon \mathfrak{a}\rightarrow \mathfrak{b}$ is an affine morphism, then it induces an unique linear map $\overrightarrow{f}$ between $T_{o}\mathfrak{a}$ and $T_{u}\mathfrak{b}$ defined by $\overrightarrow{f}(a)\coloneqq f(a)-f(o)$ for all $a\in\mathfrak{a}$. 
Conversely, if $\widehat{f}\colon T_{o}\mathfrak{a}\rightarrow T_{u}\mathfrak{b}$ is a linear map and we fix an element $b\in\mathfrak{b}$, then the map 
    \begin{align*}
f\colon \mathfrak{a}\rightarrow\mathfrak{b}\qquad a\mapsto f(a)\coloneqq \widehat{f}(a)+b=\langle \widehat{f}(a),u,b\rangle
\end{align*}
is an affine map, and also all affine maps from $\mathfrak{a}$ to $\mathfrak{b}$ arise in this way. Note that $b = f(o)$.

To make a connection with a more traditional definition of an affine space as a set $\mathfrak{a}$ with a free and transitive action of a vector space $\overrightarrow{\mathfrak{a}}$, we can set explicitly $\overrightarrow{\mathfrak{a}}= T_o\mathfrak{a}$ (or any other vector space isomorphic to a vector fibre of $\mathfrak{a}$). The action is then given by $a+v \coloneqq \langle a,o,v\rangle$, for all $a\in \mathfrak{a}$ and $v\in T_o\mathfrak{a}$. The vector from $a$ to $b$ in $\mathfrak{a}$ is then $\overrightarrow{ab}\coloneqq\langle o,a,b\rangle$.
\end{remark}

\begin{definition}\label{def.Lie}
An $\mathbb{F}$\textbf{-affgebra} (or simply, an \textbf{affgebra}) is an affine space $\mathfrak{a}$ together with a bi-affine morphism $\mu_{\mathfrak{a}}\colon\mathfrak{a}\times\mathfrak{a}\rightarrow\mathfrak{a},$
which is called the \textbf{affine product} of the affgebra. If $\mu_{\mathfrak{a}}$ is associative, then we say that $\mathfrak{a}$ is an \textbf{associative affgebra}.

If $(\mathfrak{a},\mu_{\mathfrak{a}})$ and $(\mathfrak{b},\mu_{\mathfrak{b}})$ are affgebras, then an affine morphism $f\colon \mathfrak{a}\rightarrow\mathfrak{b}$ is said to be an \textbf{affgebra morphism} if $f(\mu_{\mathfrak{a}}(a,b))=\mu_{\mathfrak{b}}(f(a),f(b))$, for all $a,b\in\mathfrak{a}$.
\end{definition}
\begin{remark}
Let $(\mathfrak{a},\mu_{\mathfrak{a}})$ be an affgebra. In view of the fact that $\mu_{\mathfrak{a}}$ is a bi-affine morphism,  the equalities
\begin{gather*}
\mu_{\mathfrak{a}}(a,\langle b,c,d\rangle)=\langle\mu_{\mathfrak{a}}(a,b),\mu_{\mathfrak{a}}(a,c),\mu_{\mathfrak{a}}(a,d)\rangle,\\
\mu_{\mathfrak{a}}(\langle a,b,c\rangle,d)=\langle\mu_{\mathfrak{a}}(a,d),\mu_{\mathfrak{a}}(b,d),\mu_{\mathfrak{a}}(c,d)\rangle
\end{gather*}
hold for all $a,b,c,d\in\mathfrak{a}$. This means that $(\mathfrak{a},\langle -,-,-\rangle,\mu_{\mathfrak{a}})$ is a \textbf{truss} \cite[Definition 3.1]{BRZ1}. 
\end{remark}
\subsection{Lie affgebras}
\begin{definition}(\cite[Definition~2.4]{BRZ3})
Let $(\mathfrak{a},\langle -,-,-\rangle, -\triangleright_{-}-)$ be an affine space.  An \textbf{affine Lie bracket on $\mathfrak{a}$} is a bi-affine morphism
\[\{-,-\}\colon\mathfrak{a}\times\mathfrak{a}\rightarrow\mathfrak{a}\] 
satisfying the following conditions: 
\begin{subequations}
\begin{align}\label{aff_antisy}
\langle \{a,b\},\{a,a\},\{b,a\}\rangle=\{b,b\} &\quad{\footnotesize\textnormal{(affine antisymmetry),}}
\\\label{aff_jac}
  \langle\{\{a,b\},c\},\{\{a,a\},a\},\{\{b,c\},a\},\{\{b,b\},b\},\{\{c,a\},b\}\rangle=\{\{c,c\},c\} &\quad{\footnotesize\textnormal{(affine Jacobi identity)}},
\end{align}
\end{subequations}
for all $a,b,c\in\mathfrak{a}$.

An affine space $\mathfrak{a}$ with an affine Lie bracket $\{-,-\}$ is called a \textbf{Lie affgebra}. 
\end{definition}
\begin{remark}
The Jacobi identity \eqref{aff_jac} differs from that in \cite[Definition~2.4]{BRZ3} by the inside placement of brackets. As noted in \cite[Remark~2.5]{BRZ3} results obtained with one version can be easily translated into another. We choose this version here to match it better with the most widespread convention used in the theory of Leibniz algebras (see Definition~\ref{def.Leibniz}).
\end{remark}
In \cite[Theorem~2.6]{BRZ3} it is indicated that a Lie affgebra is an affinization of a Lie algebra, as the unique linearization of an affine Lie bracket on  a fibre $T_{o}\mathfrak{a}$, that is,
\[[a,b]\coloneqq\{a,b\}-\{a,o\}+\{o,o\}-\{o,b\}\]
for all $a,b\in\mathfrak{a}$, makes $T_{o}\mathfrak{a}$ a Lie algebra. The Lie algebra structure on fibres are mutually isomorphic. 

Furthermore, by \cite[Theorem~3.1]{BRZ3}, the affinization of any Lie bracket as in \eqref{bullet.star}, with $\bullet$ being a Lie bracket $[-,-]$ on $T_{o}\mathfrak{a}$, makes $*$ an affine Lie bracket on $\mathfrak{a}$ if and only if, $(\lambda,-\mu,\lambda)$ is a generalized derivation in the sense of \cite{LL} on the Lie algebra $(T_{o}\mathfrak{a},[-,-])$, that is,  
\begin{equation}\label{gen.deriv}
 \lambda([a,b]) =   [\lambda(a),b]  -[a,\mu(b)], 
\end{equation}
for all $a,b\in T_{o}\mathfrak{a}$.
Irrespective of the Lie bracket, the zero maps always satisfy \eqref{gen.deriv} and, since there are no conditions on the constant $s$, every Lie algebra is a Lie affgebra and hence question {\bf Q3} from Introduction has a positive answer.

\section{Affinization of Leibniz algebras}

In this section we define a general affinization of a Leibniz algebra. 
\begin{theorem}\label{fibre_gen}
Let $(\mathfrak{a},\langle -,-,-\rangle,-\triangleright_{-}-)$ be an affine space and let $\{-,-\}: \mathfrak{a}\times \mathfrak{a}\to \mathfrak{a}$ be a bi-affine map. Define a tri-affine map
$\Lambda\colon\mathfrak{a}\times\mathfrak{a}\times\mathfrak{a}\rightarrow \mathfrak{a}$ by
\begin{equation}\label{Leibnizian}
\Lambda(a,b,c)=\langle \{a,\{b,c\}\}, \{\{a,b\},c\}, \{\{a,c\},b\}\rangle,
\end{equation}
for all $a,b,c\in\mathfrak{a}$.
For a given $o\in\mathfrak{a}$, the linearization of $\{-,-\}$ at $o$, i.e., the bracket
\[[a,b]\coloneqq\{a,b\}-\{a,o\}+\{o,o\}-\{o,b\},\]
makes  $T_{o}\mathfrak{a}$  a Leibniz algebra if and only if  the linearization of $\Lambda(a,b,c)$ at $o$ vanishes, i.e.,
\begin{equation}\label{cond_gen}
\resizebox{0.9\textwidth}{!}{$\underbrace{\Lambda(a,b,c)-\Lambda(a,b,o)-\Lambda(a,o,c)-\Lambda(o,b,c)+\Lambda(a,o,o)+\Lambda(o,b,o)+\Lambda(o,o,c)-\Lambda(o,o,o)}_{\widetilde{\Lambda}(a,b,c)}=o,$}\end{equation}
for all $a,b,c\in \mathfrak{a}$.

Furthermore, if the linearization of $\Lambda$ vanishes at $o$, its linearization at any other point of $\mathfrak{a}$ also vanishes (i.e.\ it is equal to that point).  
\end{theorem}
\begin{proof}
Note first that, in $T_{o}\mathfrak{a}$, the map $\Lambda$ \eqref{Leibnizian} takes the form:
\begin{align}\label{gen_Leib_tangent}
    \Lambda(a,b,c)=-\{\{a,b\},c\}+\{a,\{b,c\}\}+\{\{a,c\},b\}.
\end{align}
Rewriting the bilinear bracket $[-,-]$  in terms of the bi-affine bracket $\{-,-\}$ an then using \eqref{gen_Leib_tangent}, one easily finds that:
\begin{align*}
[[a,b],c]-[a,[b,c]]-[[a,c],b]=&\{\{a,b\},c\}-\{a,\{b,c\}\}-\{\{a,c\},b\}\;{\footnotesize\textnormal{($=-\Lambda(a,b,c)$)}}\\&-\{\{a,o\},c\}+\{a,\{o,c\}\}+\{\{a,c\},o\}\;{\footnotesize\textnormal{($=\Lambda(a,o,c)$)}}\\&+\{\{o,o\},c\}-\{o,\{o,c\}\}-\{\{o,c\},o\}\;{\footnotesize\textnormal{($=-\Lambda(o,o,c)$)}}\\&-\{\{o,b\},c\}+\{o,\{b,c\}\}+\{\{o,c\},b\}\;{\footnotesize\textnormal{($=\Lambda(o,b,c)$)}}\\&-\{\{a,b\},o\}+\{a,\{b,o\}\}+\{\{a,o\},b\}\;{\footnotesize\textnormal{($=\Lambda(a,b,o)$)}}\\&+\{\{a,o\},o\}-\{a,\{o,o\}\}-\{\{a,o\},o\}\;{\footnotesize\textnormal{($=-\Lambda(a,o,o)$)}}\\&-\{\{o,o\},o\}+\{o,\{o,o\}\}+\{\{o,o\},o\}\;{\footnotesize\textnormal{($=\Lambda(o,o,o)$)}}\\&+\{\{o,b\},o\}-\{o,\{b,o\}\}-\{\{o,o\},b\}\;{\footnotesize\textnormal{($=-\Lambda(o,b,o)$)}}\\=&-\widetilde{\Lambda}(a,b,c)
\end{align*}
for all $a,b,c\in \mathfrak{a}$. This proves the first statement of the theorem. 

Take now any other point $u$ and denote the linearization of $\{-,-\}$ at $u$ by $[-,-]'$. Then one easily checks that, for all $a,b \in \mathfrak{a}$,
$$
[a,b]' = \tau_o^u([\tau_u^o(a),\tau_u^o(b)]),
$$
where $\tau_o^u : T_o\mathfrak{a}\to T_u\mathfrak{a}$, $\tau_{o}^{u}(a) \coloneqq \langle a, o,u\rangle$, is the (translation) isomorphism of vector spaces. If $\tilde \Lambda$  vanishes, i.e., $[-,-]$ is a Leibniz bracket on $T_o\mathfrak{a}$,  then also $[-,-]'$ is a Leibniz bracket on $T_u\mathfrak{a}$, and hence then linearization of $\Lambda$ at $u$ vanishes by the first part of the theorem.
\end{proof}

\begin{definition}\label{Leib_def_varphi}
In the setup of Theorem~\ref{fibre_gen}, the map $\Lambda$ is called the {\bf Leibnizian} of $\{-,-\}$. The bracket $\{-,-\}$ is called an {\bf affine Leibniz bracket} if the linearization of its Leibnizian at a point $o\in \mathfrak{a}$ vanishes (and then it vanishes at any point $o\in\mathfrak{a}$ by Theorem \ref{fibre_gen}). In this case the pair $(\mathfrak{a},\{-,-\})$ is called a \textbf{Leibniz affgebra}. A \textbf{morphism of Leibniz affgebras} is an affine map preserving the affine Leibniz brackets.
\end{definition}

\begin{corollary}\label{cor.iso.fib}
  The fibres of a Leibniz affgebra are isomorphic as Leibniz algebras. 
\end{corollary}

\begin{proof}
    It follows from the proof of the second statament of Theorem~\ref{fibre_gen} that the translation map $\tau_o^u : T_o\mathfrak{a}\to T_u\mathfrak{a}$ is the required isomorphism of Leibniz algebras. 
\end{proof}
\begin{example}\label{ex.trivial}
    Let $\mathfrak{a}$ be an affine space. For any $\alpha \in \mathbb{F}$, the Leibnizian of the bracket $\{a,b\} \coloneqq \alpha \triangleright_a b$ is given by $\Lambda(a,b,c) = \alpha \triangleright_a b$. One easily checks that any linearization of $\Lambda$ vanishes, and so  $\{-,-\}$ is an affine Leibniz bracket.

    If $\sigma:\mathfrak{a}\to \mathfrak{a}$ is an affine map, then Leibnizian of the bracket $\{a,b\} \coloneqq \sigma(a)$ comes out as $\Lambda(a,b,c) = \sigma(a)$ and its linearizations vanish everywhere as needed for being an affine Leibniz bracket. Similarly, $\{a,b\} \coloneqq \sigma(b)$ is also an affine  Leibniz bracket with Leibnizian $\Lambda(a,b,c)=\langle\sigma^{2}(c),\sigma(c),\sigma(b)\rangle$.
\end{example}

\begin{theorem}\label{mainth_phi_gen}
Let $(\mathfrak{l},[-,-])$ be a Leibniz algebra and let $\lambda,\mu\colon \mathfrak{l}\rightarrow\mathfrak{l}$ be linear endomorphisms of $\mathfrak{l}$ and $s\in\mathfrak{l}$. View $\mathfrak{l}$ as the affine space $(\mathfrak{l},\langle -,-,-\rangle,-\triangleright_{-}-)$, where
\[\langle a,b,c\rangle\coloneqq a-b+c,\quad \alpha\triangleright_{a}b\coloneqq(1-\alpha)a+\alpha b\]
for all $a,b,c\in\mathfrak{l}$ and for all $\alpha\in\mathbb{F}$, and define the bi-affine bracket on $\mathfrak{l}$ by
\begin{equation}\label{th_affLeib_bracket}\{a,b\}\coloneqq[a,b]+\lambda(a)+\mu(b)+s\end{equation}
for all $a,b\in\mathfrak{l}$. The linearization at 0 of the Leibnizian $\Lambda$ of $\{-,-\}$  vanishes, i.e., $\{-,-\} $ is an affine Leibniz bracket, if and only if the following conditions hold for all $a,b,c\in\mathfrak{l}$,
\begin{subequations}\label{gen.Leibnizian}
    \begin{gather}
    \label{gen0}
    \Lambda(0,0,0) = \mu(s) +s\\
    \label{gen1}
\Lambda(a,0,0)-\Lambda(0,0,0)=[a,s]+\lambda(a),\\
   \label{gen2}
\Lambda(0,b,0)-\Lambda(0,0,0)=(\mu\lambda-\lambda\mu)(b)+\mu(b)+[s,b],\\
    \label{gen3}
\Lambda(0,0,c)-\Lambda(0,0,0)=\mu^{2}(c)-\mu(c)+\lambda\mu(c)-[s,c],\\
    \label{gen4}
\Lambda(a,b,0)-\Lambda(a,0,0)+\Lambda(0,0,0)-\Lambda(0,b,0)=[\lambda(a),b]+[a,\lambda(b)]-\lambda([a,b]),\\
    \label{gen5}
\Lambda(a,0,c)-\Lambda(a,0,0)+\Lambda(0,0,0)-\Lambda(0,0,c)=\lambda([a,c])+[a,\mu(c)]-[\lambda(a),c],\\
    \label{gen6}
\Lambda(0,b,c)-\Lambda(0,b,0)+\Lambda(0,0,0)-\Lambda(0,0,c)=\mu([b,c])+[\mu(c),b]-[\mu(b),c].
\end{gather}
\end{subequations}
 The resulting Leibniz affgebra  is denoted by $\mathfrak{a}(\mathfrak{l};\lambda,\mu,s)$. Under the previous conditions, for all $o\in\mathfrak{l}$, $T_{o}\mathfrak{a}(\mathfrak{l};\lambda,\mu,s)$ is a Leibniz algebra isomorphic to $\mathfrak{l}$.

Conversely, for any Leibniz affgebra $\mathfrak{a}$ and any point $o\in \mathfrak{a}$, there exist linear endomorphisms $\lambda,\mu$  of $T_{o}\mathfrak{a}$ and and element $s\in T_{o}\mathfrak{a}$ (necessarily satisfying conditions \eqref{gen.Leibnizian}) such that $\mathfrak{a}=\mathfrak{a}(T_{o}\mathfrak{a};\lambda,\mu,s).$ 
\end{theorem}
\begin{proof}
    Note first that in the affine space $\mathfrak{l}$ the Leibnizian takes the following form:
    \begin{equation}\label{Leibnizian_linear}
        \Lambda(a,b,c)=\{\{a,c\},b\}-\{\{a,b\},c\}+\{a,\{b,c\}\}.
    \end{equation}
    
    Assume that the linearisation at $0$ of the Leibnizian $\Lambda$ of $\{-,-\} $ vanishes. From \eqref{Leibnizian_linear} and using definition \eqref{th_affLeib_bracket} of $\{-,-\}$, the following is obtained:
    \begin{enumerate}
\item[(i)] Setting $a=b=c=0$ we obtain condition \eqref{gen0}.
        \item[(ii)] Setting $b=c=0$  we obtain condition \eqref{gen1}.
        \item[(iii)] Setting $a=c=0$ we obtain \eqref{gen2}.
        \item[(iv)] Setting $a=b=0$ we obtain \eqref{gen3}.
        \item[(v)] Setting $c=0$ and using \eqref{gen1} and \eqref{gen2} we obtain\eqref{gen4}.
    \end{enumerate}    
    Similarly, setting appropriate arguments equal to $0$ we obtain conditions \eqref{gen5} and \eqref{gen6}.
    
    Conversely, we need to show that Leibnizian $\Lambda$ vanishes at $0$ if conditions \eqref{gen.Leibnizian} hold. On the one hand, by using \eqref{Leibnizian_linear} and \eqref{th_affLeib_bracket}, we obtain that
    \begin{align*}
        \Lambda(a,b,c)&=[\lambda(a),b]+[a,\lambda(b)]-\lambda([a,b])\; {\footnotesize\textnormal{($=\Lambda(a,b,0)-\Lambda(a,0,0)+\Lambda(0,0,0)-\Lambda(0,b,0)$)}}\\
        &+\lambda([a,c])+[a,\mu(c)]-[\lambda(a),c] \;{\footnotesize\textnormal{($=\Lambda(a,0,c)- \Lambda(a,0,0)+ \Lambda(0,0,0)- \Lambda(0,0,c)$)}} \\
        &+\mu([b,c])+[\mu(c),b]-[\mu(b),c] \; {\footnotesize\textnormal{($= \Lambda(0,b,c)-\Lambda(0,b,0)+\Lambda(0,0,0)-\Lambda(0,0,c)$)}}\\
        &+[a,s]+\lambda(a)\; {\footnotesize\textnormal{($= \Lambda(a,0,0)-\Lambda(0,0,0)$)}}\\
        &+(\mu\lambda-\lambda\mu)(b)+\mu(b)+[s,b]\; {\footnotesize\textnormal{ ($=\Lambda(0,b,0)-\Lambda(0,0,0)$)}}\\
        &+\mu^2(c)-\mu(c)+\lambda\mu(c)-[s,c]\; {\footnotesize\textnormal{ ($= \Lambda(0,0,c)-\Lambda(0,0,0)$)}}\\
        &+\mu(s)+s\; {\footnotesize\textnormal{($= \Lambda(0,0,0)$)},}
    \end{align*}
    while, on the other hand, by adding and subtracting the appropriate terms, we deduce that
    \begin{align*}
        {\widetilde{\Lambda}(a,b,c)} &=\Lambda(a,b,c)\\
    &-(\Lambda(a,b,0)-\Lambda(a,0,0)+\Lambda(0,0,0)-\Lambda(0,b,0))\\
    &-(\Lambda(a,0,c)- \Lambda(a,0,0)+ \Lambda(0,0,0)- \Lambda(0,0,c))\\
    &-(\Lambda(0,b,c)-\Lambda(0,b,0)+\Lambda(0,0,0)-\Lambda(0,0,c))\\
    &-(\Lambda(a,0,0)-\Lambda(0,0,0))\\
    &-(\Lambda(0,b,0)-\Lambda(0,0,0))\\
    &-(\Lambda(0,0,c)-\Lambda(0,0,0))\\
    &-\Lambda(0,0,0).
    \end{align*}
    As a consequence, we obtain that ${\widetilde{\Lambda}(a,b,c)}=0$. 

    If conditions \eqref{gen.Leibnizian} hold, then $T_{o}\mathfrak{a}(\mathfrak{l};\lambda,\mu,s)$ is a Leibniz algebra for all $o\in\mathfrak{l}$ by Theorem \ref{fibre_gen} with bracket given by
    \[[a,b]'=[a,b]-[a,o]+[o,o]-[o,b]+o.\]
    The isomorphism of vector spaces
    \[\varphi\colon T_{o}\mathfrak{a}(\mathfrak{l};\lambda,\mu,s)\rightarrow \mathfrak{l}\qquad a\mapsto\varphi(a)\coloneqq a-o\]
    is a Leibniz algebra morphism, i.e. $\varphi([a,b]')=[\varphi(a),\varphi(b)]$; the proof of this statement is straightforward. Then, all fibres of $\mathfrak{a}(\mathfrak{l};\lambda,\mu,s)$ are mutually isomorphic and also isomorphic to $\mathfrak{l}$ as Leibniz algebras.

    Let us prove the last statement of the theorem. Let $\mathfrak{a}$ be a Leibniz affgebra and $o\in\mathfrak{a}$. By Theorem \ref{fibre_gen}, $T_{o}\mathfrak{a}$ is a Leibniz algebra with bracket defined by $[a,b]=\{a,b\}-\{a,o\}+\{o,o\}-\{o,b\}$ for all $a,b\in \mathfrak{a}$. Then, 
    \begin{align*}
        \{a,b\}=&[a,b]+\{a,o\}-\{o,o\}+\{o,b\}\\=&[a,b]+\lambda(a)+\mu(b)+s,
    \end{align*}
    where $\lambda(a)\coloneqq\{a,o\}-\{o,o\}$ and $\mu(b)\coloneqq\{o,b\}-\{o,o\}$ for all $a,b\in\mathfrak{a}$, and $s\coloneqq \{o,o\}$. Note that the maps $\{-,o\}$ and $\{o,-\}$ are affine maps from $\mathfrak{a}$ to $\mathfrak{a}$, and hence $\lambda$ and $\mu$ are linear endomorphisms of $T_{o}\mathfrak{a}$. Thus, by  construction, $\mathfrak{a}=\mathfrak{a}(T_{o}\mathfrak{a};\lambda,\mu,s)$ because both have the same affine Leibniz bracket.
\end{proof}

Theorem~\ref{mainth_phi_gen} describes all Leibniz affgebras with a prescribed Leibniz algebra fibre and all Leibniz affgebras arise in this way. Although at this point Theorem~\ref{mainth_phi_gen} is a rewriting of Theorem~\ref{fibre_gen} it becomes important and indeed crucial when it comes to studying Leibniz affgebras with a Leibnizian of a prescribed form (to which we will refer as {\bf type}); see Section~\ref{sec.types}. Then equations \eqref{gen.Leibnizian} become constraints on the form of $\lambda$, $\mu$ and $s$.

 \begin{example}[One-dimensional Leibniz affgebras] 
 \label{ex.one.dim}
     Let $\mathfrak{l}$ be a one-dimensional Leibniz algebra. Necessarily it is abelian, i.e., $[a,b]=0$ for all $a,b\in \mathfrak{l}$ and a general candidate for an affine Leibniz bracket on $\mathfrak{l}$ is of the form
     \begin{equation}\label{one.dim}
        \{a,b\} = \lambda(a)+\mu(b) + s. 
     \end{equation}
     The Leibnizian comes out as
     $$
     \Lambda(a,b,c) = \lambda(a) + \mu(b-c+s + \lambda(c) + \mu(c)) +s,
     $$
     where we have already used that any two endomorphisms of a one-dimensional vector space commute. One easily checks that this Leibnizian satisfies all conditions in Theorem~\ref{mainth_phi_gen} and all one-dimensional Leibniz affgebras arise in this form \eqref{one.dim}. Note also that each of these Leibniz affgebras is also a Lie affgebra.
 \end{example}

In the following results, we aim to characterize the homomorphisms between Leibniz affgebras $\mathfrak{a}(\mathfrak{l};\lambda,\mu,s)$. 

\begin{theorem}\label{thm.hom}
Let $(\mathfrak{l},[-,-])$ and  $(\mathfrak{l'},[-,-]')$ be Leibniz algebras and $\mathfrak{a}(\mathfrak{l};\lambda,\mu,s)$ and $\mathfrak{a}(\mathfrak{l}';\lambda',\mu',s')$ the corresponding Leibniz affgebras as in Theorem~\ref{mainth_phi_gen}. There exists a Leibniz affgebra homomorphism 
\[\varphi\colon \mathfrak{a}(\mathfrak{l};\lambda,\mu,s)\rightarrow \mathfrak{a}(\mathfrak{l}';\lambda',\mu',s')\]
if and only if there exist a Leibniz algebra morphism 
$\psi\colon\mathfrak{l}\rightarrow \mathfrak{l}'$ 
and a point $q'\in\mathfrak{l}'$ satisfying the following conditions:
\begin{subequations}\label{char_iso_conditions}
\begin{gather}\label{cond1_mor_Leib}
\psi(s)=s'+(\lambda'+\mu')(q')+\operatorname{ad}'_{q'}(q')-q',
\\\label{cond2_mor_Leib}
\psi\mu(a)=[q',\psi(a)]'+\mu'\psi(a),
\\\label{cond3_mor_Leib}
\psi\lambda(a)=(\operatorname{ad}'_{q'}+\lambda')\psi(a),
\end{gather}
\end{subequations}
for all $a\in\mathfrak{l}$.
\end{theorem}
\begin{proof}
Assume first that there exists a Leibniz affgebra homomorphism $\varphi\colon \mathfrak{a}(\mathfrak{l};\lambda,\mu,s)\rightarrow \mathfrak{a}(\mathfrak{l}';\lambda',\mu',s')$. Since $\varphi$ is an affine morphism, it induces a linear morphism $$\psi\colon\mathfrak{l}\rightarrow\mathfrak{l}'\qquad a\mapsto\psi(a)\coloneqq\varphi(a)-\varphi(0).$$  Set $q'\coloneqq \varphi(0)$, then $\varphi=\psi+q'$. Since $\varphi$ is a Leibniz affgebra homomorphism, it preserves the affine Leibniz bracket. This  implies the following chain of equivalences:
\begin{align}\label{eq9}
&\varphi(\{a,b\})=\{\varphi(a),\varphi(b)\}'\\\nonumber\iff&\psi(\{a,b\})+q'=\{\psi(a)+q',\psi(b)+q'\}'\\\nonumber\iff&\psi([a,b])+\psi\lambda(a)+\psi\mu(b)+\psi(s)+q'=[\psi(a),\psi(b)]'+[\psi(a),q']'+[q',\psi(b)]'\\\nonumber&+[q',q']'+\lambda'\psi(a)+\lambda'(q')+\mu'\psi(b)+\mu'(q')+s'.
\end{align}
By setting $a=0=b$ in \eqref{eq9} and using the linearity of the maps involved, equality  \eqref{cond1_mor_Leib} is obtained. If we replace in \eqref{eq9} the value of $\psi(s)$ obtained in \eqref{cond1_mor_Leib}, then 
\begin{align}\label{eq10}
\psi([a,b])+\psi\lambda(a)+\psi\mu(b)=[\psi(a),\psi(b)]'+[\psi(a),q']'+[q',\psi(b)]'+\lambda'\psi(a)+\mu'\psi(b).
\end{align}
By setting $a=0$ and $b=0$ in \eqref{eq10} we obtain equalities \eqref{cond2_mor_Leib} and \eqref{cond3_mor_Leib}, respectively. The fact that $\psi$ preserves the Leibniz bracket follows by using \eqref{cond2_mor_Leib} and \eqref{cond3_mor_Leib} in \eqref{eq10}.

Conversely, assume that there exist  a Leibniz algebra morphism $\psi\colon \mathfrak{l}\rightarrow \mathfrak{l}'$ and a point $q'\in\mathfrak{l}'$ such that conditions \eqref{char_iso_conditions} hold. Define $\varphi(a)=\psi(a)+q'$ and let us confirm that $\varphi$ is a Leibniz affgebra homomorphism between $\mathfrak{a}(\mathfrak{l};\lambda,\mu,s)$ and $\mathfrak{a}(\mathfrak{l}';\lambda',\mu',s')$. Since $\psi$ is a linear map, it follows that $\varphi$ is an affine homomorphism. 
Furthermore, for all $a,b\in \mathfrak{l}$,
$$
\begin{aligned}
    \varphi(\{a,b\}) &=\psi([a,b])+\psi\lambda(a)+\psi\mu(b)+\psi(s)+q'
    \\ &=[\psi(a),\psi(b)]'+[\psi(a),q']'+\lambda'\psi(a)+[q',\psi(b)]'+\mu'\psi(b)+[q',q']'+\lambda'(q')+\mu'(q')+s'\\&=\{\varphi(a),\varphi(b)\}',
\end{aligned}
$$
where the second equality follows by the fact that $\psi$ preserves the Leibniz bracket and equalities \eqref{char_iso_conditions}.
\end{proof}
\begin{corollary}\label{iso_charac}
Leibniz affgebras $\mathfrak{a}(\mathfrak{l};\lambda,\mu,s)$ and $\mathfrak{a}(\mathfrak{l}';\lambda',\mu',s')$ are mutually isomorphic if and only if there exist an isomorphism of Leibniz algebras $\psi\colon \mathfrak{l}\rightarrow\mathfrak{l}'$ and a point $q\in\mathfrak{l}$ satisfying the following conditions:
\begin{subequations}\label{char_iso_conditions2}
\begin{gather}
\label{cond1_iso_Leib}
s'=\psi(s-(\lambda+\mu)(q)+\operatorname{ad}_{q}(q)+q),
\\\label{cond2_iso_Leib}
\mu'=\psi(\mu-[q,-])\psi^{-1},
\\\label{cond3_iso_Leib}
\lambda'=\psi(\lambda-\operatorname{ad}_{q})\psi^{-1}.
\end{gather}
\end{subequations}
\end{corollary}
\begin{proof}
If $\varphi\colon \mathfrak{a}(\mathfrak{l};\lambda,\mu,s)\rightarrow \mathfrak{a}(\mathfrak{l}';\lambda',\mu',s')$ is an isomorphism of Leibniz affgebras, then there exists  an isomorphism of Leibniz algebras $\psi(a)=\varphi(a)-q'\colon\mathfrak{l}\rightarrow\mathfrak{l}'$, where $q'=\varphi(0)$, satisfying conditions \eqref{char_iso_conditions} (note that the inverse of $\psi$ is given by $\psi^{-1}(x)\coloneqq\varphi^{-1}(x+q')$ for all $x\in\mathfrak{l}'$). Setting $q\coloneqq\psi^{-1}(q')\in\mathfrak{l}$ and taking into account that 
\begin{gather*}\label{eq11}
\operatorname{ad}'_{q'}=\psi \operatorname{ad}_{q}\psi^{-1},\quad [q',-]'=\psi [q,-]\psi^{-1},
\end{gather*}
we quickly arrive at \eqref{char_iso_conditions2} as a reformulation of \eqref{char_iso_conditions}.
\end{proof}

\begin{corollary}
Let $\mathfrak{a}$ and $\mathfrak{b}$ be Leibniz affgebras. If their fibres are non-isomorphic as Leibniz algebras, then $\mathfrak{a}$ and $\mathfrak{b}$ are non-isomorphic.
\end{corollary}
\begin{proof}
By Theorem~\ref{mainth_phi_gen}, for a point $o\in\mathfrak{a}$ and a point $u\in\mathfrak{b}$, there exist  linear endomorphisms $\lambda,\mu$ of $T_{o}\mathfrak{a}$, and an element $s\in T_{o}\mathfrak{a}$, $\lambda',\mu'$ linear endomorphisms in $T_{u}\mathfrak{b}$ and $s'\in T_{u}\mathfrak{b}$ such that
\[\mathfrak{a}=\mathfrak{a}(T_{o}\mathfrak{a};\lambda,\mu,s)\textnormal{ and }\mathfrak{b}=\mathfrak{a}(T_{u}\mathfrak{b};\lambda',\mu',s').\]
By the previous theorem, if $\mathfrak{a}(T_{o}\mathfrak{a};\lambda,\mu,s)$ and $\mathfrak{a}(T_{u}\mathfrak{b};\lambda',\mu',s')$ are isomorphic, then there exists a Leibniz algebra isomorphism $\psi\colon T_{o}\mathfrak{a}\rightarrow T_{u}\mathfrak{b}$ as required.
\end{proof}

The next proposition provides one with a characterization of subaffgebras of a Leibniz affgebra $\mathfrak{a}(\mathfrak{l};\lambda,\mu,s)$.
\begin{proposition} Let $\mathfrak{a}(\mathfrak{l};\lambda,\mu,s)$ be a Leibniz affgebra and let $\mathfrak{b}$ be an affine subspace of $\mathfrak{a}(\mathfrak{l};\lambda,\mu,s)$.  The following assertions are equivalent:
\begin{enumerate}
\item[(i)] $\mathfrak{b}$ is a Leibniz subaffgebra of $\mathfrak{a}(\mathfrak{l};\lambda,\mu,s)$,
\item[(ii)] $\mathfrak{b}=a+\mathfrak{h}$ with $\mathfrak{h}$ a Leibniz subalgebra of $\mathfrak{l}$ such that the following conditions hold:
\begin{subequations}\label{cond_subaff}
\begin{gather}
\label{cond1_subaff_Leib} [a,a]-a+(\lambda+\mu)(a)+s\in\mathfrak{h},
\\\label{cond2_subaff_Leib} (\lambda+\operatorname{ad}_{a})(\mathfrak{h})\subseteq \mathfrak{h},
\\\label{cond3_subaff_Leib} (\mu+[a,-])(\mathfrak{h})\subseteq \mathfrak{h}.
\end{gather}
\end{subequations}
\end{enumerate}

Furthermore, under the previous conditions, $\mathfrak{b}=a+\mathfrak{h}$ is a Leibniz affgebra isomorphic to $$\mathfrak{a}(\mathfrak{h};\lambda+\operatorname{ad}_{a},\mu+[a,-],[a,a]-a+(\lambda+\mu)(a)+s).$$
\end{proposition}
\begin{proof}
Any affine subspace of $\mathfrak{a}(\mathfrak{l};\lambda,\mu,s)$ is of the form $a+\mathfrak{h}$ with $\mathfrak{h}$ a vector subspace of $\mathfrak{l}$ and $a\in \mathfrak{a}(\mathfrak{l};\lambda,\mu,s)$. Then, set $\mathfrak{b}=a+\mathfrak{h}$.

(i) $\implies$ (ii). If $\mathfrak{b}$ is a Leibniz subaffgebra of $\mathfrak{a}(\mathfrak{l};\lambda,\mu,s)$, then $\{\mathfrak{b},\mathfrak{b}\}\subseteq \mathfrak{b}$. Hence, for all $a+u,a+v\in\mathfrak{b}$, 
\begin{align*}
\{a+u,a+v\}=[a,a]+[a,v]+[u,a]+[u,v]+\lambda(a+u)+\mu(a+v)+s\in\mathfrak{b}.
\end{align*}
From the above equation, the following consequences are obtained:
\begin{itemize}
\item By setting $u=0=v$, $[a,a]+(\lambda+\mu)(a)+s\in a+\mathfrak{h}$ which implies \eqref{cond1_subaff_Leib}.
\item By setting $v=0$, $[a,a]+[u,a]+\mu(a)+\lambda(a+u)+s\in a+\mathfrak{h}$. Then, by applying \eqref{cond1_subaff_Leib}, \eqref{cond2_subaff_Leib} is obtained.
\item By setting $u=0$, $[a,a]+[a,v]+\lambda(a)+\mu(a+v)+s\in a+\mathfrak{h}$. And then, again applying \eqref{cond1_subaff_Leib}, \eqref{cond3_subaff_Leib} is obtained.
\item By \eqref{cond_subaff} it automatically follows that $\mathfrak{h}$ is a Leibniz subalgebra of $\mathfrak{l}$.
\end{itemize}

(ii) $\implies$ (i). Assume that $\mathfrak{b}=a+\mathfrak{h}$ with $\mathfrak{h}$ a Leibniz subalgebra of $\mathfrak{l}$ such that the conditions \eqref{cond_subaff} hold. Then:
\begin{align*}
&\{a+u,a+v\}=\underbrace{[a,a]+(\lambda+\mu)(a)+s}_{\in\mathfrak{b}\,\textnormal{(by \eqref{cond1_subaff_Leib})}}+\underbrace{[a,v]+\mu(v)}_{\in\mathfrak{h}\,\textnormal{(by \eqref{cond3_subaff_Leib})}}+\underbrace{[u,a]+\lambda(u)}_{\in\mathfrak{h}\,\textnormal{(by \eqref{cond2_subaff_Leib})}}+\underbrace{[u,v],}_{\in\mathfrak{h}\,\textnormal{(by $[\mathfrak{h},\mathfrak{h}]\subseteq\mathfrak{h}$)}}
\end{align*}
so that $\{a+u,a+v\}\in\mathfrak{b}$.

The final assertion follows by restricting the identity automorphism of $\mathfrak{l}$ to $\mathfrak{h}$ and using Corollary~\ref{iso_charac} with $q=a$.
\end{proof}

The problem of affinization of Leibniz algebras diverges into two  directions. The first direction is to find and classify all affine Leibniz brackets on a given affine space. The initial step is easy, as Theorem~\ref{fibre_gen} and Theorem~\ref{mainth_phi_gen} show that a datum consisting of a Leibniz bracket on any of the fibres together with an element and a pair of linear endomorphisms of this fibre gives rise to an affine Leibniz bracket. The next step is to use the classification of Leibniz affgebras up to isomorphism and use Corollary~\ref{iso_charac} to resolve the relations between different linear data.
The second direction involves attempting to classify Leibniz affgebras with Leibnizian of a prescribed type.   By Theorem~\ref{mainth_phi_gen}, the form of a chosen Leibnizian restricts the range of Leibniz algebra structures on the fibres of Leibniz affgebras with this Leibnizian. In the next section we look at a few possible choices, each one coming with their own motivation, and discuss their advantages and shortcomings.

Note that Theorem~\ref{thm.hom} and Corollary~\ref{iso_charac}, which facilitate the classification of Leibniz affgebras of a given type, depend implicitly on the type of Leibniz affgebra as the choice of the Leibnizian affects the conditions \eqref{gen.Leibnizian} that $\lambda$, $\mu$ and $s$ have to satisfy.

\section{Types of Leibniz affgebras}\label{sec.types}
\subsection{Derivative and vector-type Leibniz affgebras}\label{sec.inhom}
In the linear setup a derivation of a function at a given point can be understood as a vector measure of the speed of change of the function at this point. In purely affine or vector space independent setting this should be replaced by a translation from a value of the function at a given point by the vector of the change of value. The resulting Leibniz rule for an affine derivation $D$ of an affgebra $\mathfrak{a}$ with the multiplication $*$ comes out as
\begin{equation}\label{aff.Leibniz}
    D(a*b) = \langle D(a)*b,a*b,a*D(b)\rangle, 
\end{equation}
for all $a,b\in \mathfrak{a}$ (see also \cite[Definition 3.10]{BRZ2}). To translate the derivation property of a Leibniz bracket note first that \eqref{Leibnizian} can be equivalently stated as
\begin{equation}\label{Leibnizian.der}
   \{\{a,b\},c\} = \langle \{\{a,c\},b\}, \Lambda(a,b,c),\{a,\{b,c\}\}\rangle.
\end{equation}
Thus the maps $\{-,c\}:\mathfrak{a} \to \mathfrak{a}$ satisfy the affine derivation property \eqref{aff.Leibniz} provided $\Lambda(a,b,c)= \{a,b\}$, for all $a,b,c,\in \mathfrak{a}$. The following lemma is proven by a simple verification.
\begin{lemma}\label{lem.Leib.der}
 Let $\{-,-\}$ be a bi-affine operation of $\mathfrak{a}$. Then, for all $o\in \mathfrak{a}$  the tri-affine map $\Lambda: \mathfrak{a}^3\to \mathfrak{a}$, $(a,b,c)\mapsto \{a,b\}$ satisfies \eqref{cond_gen}.  
\end{lemma}
%
%
\begin{definition}\label{inhomLeibnizAffgebra}
A Leibniz affgebra $(\mathfrak{a},\{-,-\})$ with  Leibnizian $\Lambda(a,b,c) = \{a,b\}$ is called a \textbf{derivative Leibniz affgebra}.
\end{definition}

\begin{example}\label{ex.L.deriv}
All types of Leibniz affgebras in Example~\ref{ex.trivial} are derivative Leibniz affgebras.

Let $(\mathfrak{a},\mu_{\mathfrak{a}})$ be an associative $\mathbb{F}$-affgebra and let $D\colon \mathfrak{a}\rightarrow \mathfrak{a}$ be an affine map such that
\begin{equation}\label{condD}
D(D(a)b)=D(a)D(b)=D(aD(b))
\end{equation}
for all $a,b\in\mathfrak{a}$. One can check by direct though quite tedious computation that $\mathfrak{a}$ is a derivative Leibniz affgebra with the affine Leibniz bracket given by
$\{a,b\}\coloneqq\langle aD(b), D(b)a,a \rangle$. Note that any idempotent $D$ that is either an affgebra morphism or an affine derivation of $\mathfrak{a}$ satisfies \eqref{condD}.
\end{example}

Construction of derivative Leibniz affgebras with a prescribed Leibniz algebra fibre $(\mathfrak{l},[-,-])$ via the formula \eqref{th_affLeib_bracket} boils down to solving the following equations that result from equations \eqref{gen.Leibnizian} in Theorem~\ref{mainth_phi_gen}:
\begin{subequations}\label{gen.Leibnizian.der}
    \begin{align}
    \label{gen12.der}
\mu(s) = 0,&  \qquad [a,s]=0, \qquad 
(\mu\lambda-\lambda\mu)(b)+[s,b] =0,\\
    \label{gen34.der}
\mu^{2}(c)-\mu(c)+\lambda\mu(c)-[s,c]=0,& \qquad
[a,b]=[\lambda(a),b]+[a,\lambda(b)]-\lambda([a,b]),\\
    \label{gen56}
\lambda([a,c])+[a,\mu(c)]-[\lambda(a),c] =0,&\qquad
  \mu([b,c])+[\mu(c),b]-[\mu(b),c] =0,
\end{align}
\end{subequations}
for all $a,b,c\in\mathfrak{l}$. These are restrictive, but we will not comment on them here, except for pointing out that they include equations describing {\em homogeneous} Leibniz affgebras that are to be introduced in the forthcoming Section~\ref{sec.hom}. Instead we describe the connection with Leibniz affgebras that can be defined as a noncommutative version of Lie affgebras as understood by Grabowska, Grabowski and Urba\'{n}ski in \cite{GGU}, i.e.\  affine spaces over a fixed vector space with a vector valued Lie bracket. As explained in \cite[Proposition 3.5]{BRZ2} these correspond to Lie affgebras (in the sense of Definition~\ref{def.Lie}) with idempotent affine Lie brackets.

\begin{definition}
Let $\mathfrak{a}$ be an affine space over a vector space $\overrightarrow{\mathfrak{a}}$. A \textbf{vector-valued Leibniz bracket} on $(\mathfrak{a},\overrightarrow{\mathfrak{a}})$ is a bi-affine map 
\[[-,-]_{v}\colon \mathfrak{a}\times \mathfrak{a}\rightarrow\overrightarrow{\mathfrak{a}}\]
such that $[-,-]_{v}$ satisfies the Leibniz rule on $\overrightarrow{\mathfrak{a}}$, i.e., for all $a,b,c\in \mathfrak{a}$,
\begin{equation}\label{vector-valued-LeibRule}
    [[a,b]_{v},\overrightarrow{c}]_{v}=[[a,c]_{v},\overrightarrow{b}]_{v}+[\overrightarrow{a},[b,c]_{v}]_{v}
\end{equation}
where $[\overrightarrow{a},-]_{v},[-,\overrightarrow{b}]_{v}: \overrightarrow{\mathfrak{a}}\to \overrightarrow{\mathfrak{a}}$ are the linear parts of the affine maps $[a,-]_{v},[-,b]_v:\mathfrak{a}\to \mathfrak{a}$, respectively, for all $a,b\in\mathfrak{a}$.
\end{definition}

\begin{remark}
Note that by setting $c=b$ in \eqref{vector-valued-LeibRule} we obtain
\begin{equation}\label{quasi-nilpotent}
    [\overrightarrow{a},[b,b]_{v}]_{v} =0,
\end{equation}
for all $a,b\in \mathfrak{a}$, which has the same flavour as \eqref{condLeibR} in Remark~\ref{rem.condLeibR}.
\end{remark}
\begin{proposition}
    Let $\mathfrak{a}$ be an affine space. There exists a bijective correspondence between vector-valued Leibniz brackets on $(\mathfrak{a},T_{o}\mathfrak{a})$ and derivative affine Leibniz brackets on $\mathfrak{a}$ (following Definition \ref{inhomLeibnizAffgebra}) for all $o\in \mathfrak{a}$.
\end{proposition}
\begin{proof}
First observe that in terms of the binary operations in $T_{o}\mathfrak{a}$, the derivative Leibniz bracket condition reads as
\begin{equation}\label{deriv.add}
    \{\{a,b\},c\} =  \{\{a,c\},b\} - \{a,b\}+\{a,\{b,c\}\}.
\end{equation}

    Let us assume that $\{-,-\}$ is a derivative affine Leibniz bracket on $\mathfrak{a}$ and, for $o \in \mathfrak{a}$, define the bi-affine map
    \begin{align*}
        [-,-]_{v}\colon \mathfrak{a}\times \mathfrak{a} \rightarrow T_{o}\mathfrak{a}, \qquad(a,b)\mapsto [a,b]_{v}\coloneqq \{a,b\}-a=\langle\{a,b\},a,o\rangle. 
    \end{align*}
    Note that 
    \begin{align*}
        [x,\overrightarrow{a}]_{v}= \{x,a\}-x-\{o,a\} \quad \textnormal{and} \quad 
        [\overrightarrow{a},x]_{v}=\{a,x\} -\{a,o\}.
    \end{align*}
    As a consequence,
    \begin{align*}
        [[a,b]_{v},\overrightarrow{c}]_{v}&=\{\{a,b\},c\}-\{a,c\}+a-\{a,b\} \quad \textnormal{and} \quad [\overrightarrow{a},[b,c]_{v}]_{v}=\{a,\{b,c\}\} -\{a,b\}.
    \end{align*}
Putting all this together and using \eqref{deriv.add} one easily verifies \eqref{vector-valued-LeibRule}.

In the other direction, for $o\in \mathfrak{a}$, let us assume that $[-,-]_{v}$ is a vector-valued Leibniz bracket on $(\mathfrak{a},\overrightarrow{\mathfrak{a}} = T_o\mathfrak{a})$ and define the bi-affine map
    \begin{align*}
        \{-,-\}\colon \mathfrak{a}\times \mathfrak{a}\rightarrow \mathfrak{a},\qquad (a,b) \mapsto\{a,b\}\coloneqq[a,b]_{v}+a=\langle[a,b]_{v},o,a\rangle.
    \end{align*}
    Noting that
    \begin{align*}
        \{\{a,b\},c\}=[[a,b]_{v},\overrightarrow{c}]_{v}+[a,b]_{v}+[a,c]_{v}+a \quad \textnormal{and} \quad  \{a,\{b,c\}\}=[\overrightarrow{a},[b,c]_{v}]_{v}+[a,b]_{v}+a,
    \end{align*}
    we obtain \eqref{deriv.add}, hence $\{-,-\}$ is a derivative affine Leibniz bracket on $\mathfrak{a}$. 
\end{proof}

\subsection{Homogeneous Leibniz affgebras}\label{sec.hom}
The form of \eqref{Leibnizian} immediately implies that a Leibnizian $\Lambda$ on $\mathfrak{a}$ necessarily satisfies the condition $\Lambda(a,b,b) = \{a,\{b,b\}\}$ for all $a,b\in \mathfrak{a}$. Taking heed of this one can consider the following definition. 

\begin{definition}\label{homLeibnizAffgebra}
Let $\mathfrak{a}$ be an affine space. An affine Leibniz bracket $\{-,-\}$ on $\mathfrak{a}$  with Leibnizian $\Lambda(a,b,c) = \{a,\{b,b\}\}$, for all $a,b,c\in \mathfrak{a}$, is said to be \textbf{homogeneous}.  A \textbf{homogeneous Leibniz affgebra} is an affine space $\mathfrak{a}$ together with a homogeneous affine Leibniz bracket.

\end{definition}
\begin{remark}
Assume that $(\mathfrak{a},\{-,-\})$ is a derivative Leibniz affgebra and $o\in \mathfrak{a}$. Then, by \eqref{deriv.add}, in $T_{o}\mathfrak{a}$ we have that:
\begin{align*}
\{\{a,b\},b\}=\{a,\{b,b\}\}-\{a,b\}+\{\{a,b\},b\}\iff\{a,b\}=\{a,\{b,b\}\}
\end{align*}
for all $a,b\in\mathfrak{a}$. As a consequence, any derivative Leibniz affgebra is also homogeneous. 
\end{remark}
\begin{example}\label{ex.L.hom}
Every Leibniz algebra $(\mathfrak{l},[-,-])$ can be seen as a homogeneous Leibniz affgebra with the affine Leibniz bracket defined as follows:
\[\{a,b\}\coloneqq [a,b]+a+s\]
for all $a,b\in\mathfrak{l}$ and $s\in \mathcal{LZ}(\mathfrak{l})$.
Indeed, 
\begin{align*}
\Lambda(a,b,c) =&-[[a,b],c]+[a,[b,c]]+[[a,c],b] +[a,b]+ [s,b-c] +[a,s]+a +s \\=&[a,b]+[a,s]+a +s\;\footnotesize\textnormal{(by \eqref{LeibRule} and $s\in\mathcal{LZ}(\mathfrak{l})$)}\\=&\{a,\{b,b\}\}\;\footnotesize\textnormal{(by \eqref{condLeibR})}.
\end{align*}
Note that these homogeneous Leibniz affgebras are not derivative.
\end{example}

Leibniz affgebras in Example~\ref{ex.L.hom} are of course a special case of general construction of Leibniz affgebras with a prescribed Leibniz fibre described in Theorem~\ref{mainth_phi_gen}.

\begin{proposition}
    \label{imp_Leib}
Let $(\mathfrak{l},[-,-])$ be a Leibniz algebra. Then  $\mathfrak{l}$ is a homogeneous Leibniz affgebra as in Theorem~\ref{mainth_phi_gen} if and only if 
\begin{subequations}\label{Leib_cond_th}
\begin{align}\label{Leib_cond_th1}
[\lambda(a),b]-[a,\mu(b)]=\lambda([a,b]),\\
\label{Leib_cond_th2}
[\mu(a),b]-[\mu(b),a]=\mu([a,b]),\\
\label{Leib_cond_th3}
[s,a]+\mu(a)=\mu^{2}(a)+\lambda\mu(a)
\end{align}
\end{subequations}
hold. 
All homogeneous Leibniz affgebras with fibres isomorphic to $(\mathfrak{l},[-,-])$ arise in this way.
%
\end{proposition}
\begin{proof}
First note that if the affine Leibniz bracket is given by \eqref{th_affLeib_bracket}, then
$$
\{a,\{b,b\}\} = [a,\lambda(b) +\mu(b) + s] + \lambda(a) + \mu([b,b]) + \mu^2(b) +\mu\lambda(b) +\mu(s) +s,
$$
where we used \eqref{condLeibR}. Conditions \eqref{gen0} and \eqref{gen1} do not generate any constraints on $\lambda,\mu$ and $s$. Since the homogeneous Leibnizian does not depend on the third argument, the left hand sides of equations \eqref{gen3}, \eqref{gen5} and \eqref{gen6} vanish and these equations lead to all three \eqref{Leib_cond_th}. In view of \eqref{Leib_cond_th3}, \eqref{gen2}  yields $\mu([a,a]) =0$ which is already contained in \eqref{Leib_cond_th2} by setting $a=b$. Finally, \eqref{gen4} coincides with \eqref{Leib_cond_th1}. The assertions then follow by Theorem~\ref{mainth_phi_gen}.
\end{proof}

The main advantage of homogeneous Leibniz affgebras is a relative simplicity of conditions \eqref{Leib_cond_th}. This on one hand allows one to construct wide classes of examples, while on the other also permits effective classification of such affgebras. We illustrate this in detail in what follows.

\begin{example}[1-dimensional homogeneous Leibniz affgebras]\label{example.1dim}
Let $(\mathfrak{l},[-,-])$ be a 1-dimensional Leibniz algebra with basis $\{e\}.$ By \eqref{LeibRule}, $\mathfrak{l}$ is automatically abelian, i.e., $[\mathfrak{l},\mathfrak{l}]=0$. Consider $\lambda,$ $\mu$ linear endomorphisms of $\mathfrak{l}$ given by $\lambda(e)=\lambda e$, $\mu(e)=\mu e$ with $\lambda,\mu\in\mathbb{F}$. Due to $\mathfrak{l}$ being abelian, conditions \eqref{Leib_cond_th1} and \eqref{Leib_cond_th2} do not induce any constraints on $\lambda$, $\mu$. However, \eqref{Leib_cond_th3} holds if and only if $\mu=0$ or $\mu=1-\lambda$. As a consequence, there are two possible families of Leibniz affgebras fibred by $\mathfrak{l}$: $\mathfrak{a}(\mathfrak{l};\lambda,0,s)$ and $\mathfrak{a}(\mathfrak{l};\lambda,\mathrm{id}-\lambda,s)$ with no restrictions on $s$. Let $\psi(e)=\psi e$, with $\psi\in\mathbb{F}\setminus\{0\}$, be an automorphism of $\mathfrak{l}$ and $q=q e\in\mathfrak{l}$ with $q\in\mathbb{F}$. By applying Corollary \ref{iso_charac} with $s=se$, $s\in\mathbb{F}$, it is obtained that 
\begin{gather}\label{s'_1dim}
    s'=\psi(s-q(\lambda+\mu-1))e,
\end{gather}
whereas $\lambda'=\lambda$ and $\mu'=\mu$. 

On the one hand, in case $\mu=0$, $s'=\psi(s-q(\lambda-1))e$ by \eqref{s'_1dim}. If $\lambda\neq 1$, then $s'=0$ by choosing  $q=\frac{s}{\lambda-1}$. However, if $\lambda=1$, then $s'=\psi s e$ which can be either $s'=0$ when $s=0$, or $s'=e$ by setting $\psi=s^{-1}$ when $s\neq 0$. On the other hand, when $\mu=1-\lambda$, $s'=\psi s e$ by \eqref{s'_1dim}. Then we conclude that one-dimensional homogeneous Leibniz affgebras fall into three classes:
$$
\mathfrak{a}(\mathfrak{l};\lambda,0,0), \qquad \mathfrak{a}(\mathfrak{l};\lambda,\mathrm{id}-\lambda,0) , \qquad \mathfrak{a}(\mathfrak{l};\lambda,\mathrm{id}-\lambda,e),
$$
where $e$ is a fixed element of a basis for $\mathfrak{l}$.

\end{example}

\begin{example}[Homogeneous Leibniz affgebras with abelian fibres] \label{example_abelian_fibre}
Let $\mathfrak{l}$ be the abelian Leibniz algebra, i.e., $[\mathfrak{l},\mathfrak{l}]=0$. Then \eqref{Leib_cond_th1} and \eqref{Leib_cond_th2} always hold, while \eqref{Leib_cond_th3} is equivalent to the functorial identity
$$
(\mu +\lambda - \mathrm{id})\mu =0,
$$
which can be solved for $\lambda$ as
$$
\lambda (a) = \begin{cases}
    \nu(a), & a \in \ker \mu,\\
    \mathrm{id}-\mu, & a\not\in \ker \mu,
\end{cases}
$$
where $\nu: \ker \mu \to \mathfrak{l}$ is any linear function. Consequently,
$$
\{a,b\} = \begin{cases}
    \nu(a) +\mu(b) +s, & a \in \ker \mu,\\
    a+\mu(b-a) +s, & a\not\in \ker \mu.
\end{cases} $$
\end{example}

Before turning to examples in dimensions equal or greater than 2, it will be essential to study a highly relevant subclass of Leibniz affgebras: homogeneous Leibniz affgebras with a Lie fibre.

\begin{lemma}\label{lie_fibre_conseq}
If $\mathfrak{a}$ is a homogeneous Leibniz affgebra with Lie fibre $(\mathfrak{g},[-,-])$, i.e., $\mathfrak{a}\simeq\mathfrak{a}(\mathfrak{g};\lambda,\mu,s)$, then the following statements hold:
\begin{enumerate}
\item[(i)] $\mu$ is a derivation of $\mathfrak{g}$.
\item[(ii)] $\kappa\coloneqq\lambda+\mu\in C(\g)$, where
$$
C(\g) = \{\kappa \in \mathrm{Lin}(\g)\; |\; \kappa ([a,b]) = [\kappa(a),b], \; \mbox{for all}\; a,b\in \g\}
$$
is the \textbf{centroid} of $\g$; see \cite[Chapter~X]{Jac:Lie}. This implies that
\begin{enumerate}
    \item[(ii.1)] $\kappa^{n}\in C(\mathfrak{g})$ for all $n\in\mathbb{N}$.
    \item[(ii.2)] For all $a\in\mathfrak{g}$, $\mathfrak{g}_{\kappa}(a)\coloneqq\operatorname{span}_{\mathbb{F}}\{\kappa^{n}(a)\colon n\in\mathbb{N}\}$ are abelian Lie subalgebras of $\mathfrak{g}$.
\end{enumerate}
\end{enumerate}
\end{lemma}
\begin{proof}
Statement (i) is a direct consequence of \eqref{Leib_cond_th2} and antisymmetry of the Lie bracket $[-,-]$. Then, for all $a,b\in\g$,
\begin{align*}
[\kappa(a),b]=[\lambda(a),b]+[\mu(a),b]=\lambda([a,b])+[a,\mu(b)]+[\mu(a),b]=\lambda([a,b])+\mu([a,b])=\kappa([a,b]),
\end{align*}
i.e., $\kappa\in C(\g)$, where the second equality follows by \eqref{Leib_cond_th1} and the third equality is a consequence of the fact that $\mu\in\operatorname{Der}(\g)$. The remaining assertions are well-known (see e.g.\ \cite[Lemma~3.2]{BRZ3}).
\end{proof}
\begin{theorem}\label{semisimplefibre}
Let $\mathbb{F}$ be an algebraically closed field. If $(\g,[-,-])$ is a simple Lie algebra, then every homogeneous Leibniz affgebra $\mathfrak{a}$ with Lie fibre  $\g$ is isomorphic to $\mathfrak{a}(\g;\alpha\,\mathrm{id},0,0)$ for some $\alpha\in\mathbb{F}$. 
\end{theorem}
\begin{proof}
Let $\mathfrak{a}\simeq\mathfrak{a}(\g;\lambda,\mu,s)$. By Corollary \ref{iso_charac} and antisymmetry of $[-,-]$, $\mathfrak{a}(\g;\lambda,\mu,s)$ is isomorphic to $\mathfrak{a}(\g;\lambda',\mu',s')$ if and only if
\begin{gather}\label{eq12eq}
s'=s-(\lambda+\mu)(q)+q,\qquad
\mu'=\mu+\operatorname{ad}_{q},\qquad
\lambda'=\lambda-\operatorname{ad}_{q},
\end{gather}
for $q\in\g$. Since $\g$ is simple and $\mu$ is a derivation,  $\mu=\operatorname{ad}_{z}$ for some $z\in\g$ because  all derivations of a simple Lie algebra are inner derivations. By setting $q=-z$, equations \eqref{eq12eq} become
\begin{gather*}\label{eq15eq}
s'=s+\lambda(z)-z,\qquad
\mu'=0,\qquad\lambda'=\lambda+\mu=\kappa.
\end{gather*}
By (ii) of Lemma \ref{lie_fibre_conseq}, $\kappa\in C(\g)$, so that $\kappa\circ \operatorname{ad}_{a}=\operatorname{ad}_{a}\circ\kappa$, for all $a\in\g$. Then, since $\mathbb{F}$ is an algebraically closed field, Schur's Lemma affirms that $\kappa=\alpha\, \mathrm{id}$ for some $\alpha\in\mathbb{F}$, and then 
$\lambda'=\kappa=\alpha\,\mathrm{id}.$
Furthermore, since the Leibniz affgebra $\mathfrak{a}(\g;\alpha\,\mathrm{id},0,s')$ satisfies \eqref{Leib_cond_th3}, we obtain that $s'\in\mathcal{Z}(\g)$. However, $\mathcal{Z}(\g)=0$ because $\g$ is simple, and so $s'=0$. Thus, we conclude that
\[\mathfrak{a}\simeq\mathfrak{a}(\g;\lambda,\mu,s)\simeq\mathfrak{a}(\g;\alpha\,\mathrm{id},0,0),\]
as asserted.
\end{proof}
\begin{proposition}\label{nontriv_subal}
Let $\mathfrak{a}$ be a homogeneous Leibniz affgebra with a Lie fibre $(\g,[-,-])$ which does not contain non-trivial abelian subalgebras. Then, $\mathfrak{a}$ is isomorphic to $\mathfrak{a}(\mathfrak{g};\alpha\, \mathrm{id}-\mu,\mu,s)$,
$\alpha\in\mathbb{F}$.
\end{proposition}
\begin{proof}
Let $\mathfrak{a}\simeq\mathfrak{a}(\mathfrak{g};\lambda,\mu,s)$ with $\{a_{i}\}$ a basis for $\g$. By Lemma \ref{lie_fibre_conseq}, the subalgebras
$\g_{\kappa}(a_{i})$ of $\g$ are abelian, and then, as a consequence of the fact that $\g$ does not contain non-trivial abelian subalgebras, $\g_{\kappa}(a_{i})$ are 1-dimensional. Hence, $\g_{\kappa}(a_{i})=\mathbb{F}a_{i}$, and so $\kappa(a_{i})=\kappa_{i}a_{i}$,  for some $\kappa_{i}\in\mathbb{F}$. In addition, by Lemma \ref{lie_fibre_conseq} we know that $\kappa\in C(\g)$, and then 
\begin{align}\label{qz}
\kappa_{i}[a_{i},a_{j}]=&[\kappa_{i}a_{i},a_{j}]=[\kappa(a_{i}),a_{j}]=\kappa([a_{i},a_{j}])\\\nonumber=&-\kappa([a_{j},a_{i}])=-[\kappa(a_{j}),a_{i}]=-[\kappa_{j}a_{j},a_{i}]=-\kappa_{j}[a_{j},a_{i}]=\kappa_{j}[a_{i},a_{j}].
\end{align}
Note that if $i\neq j$, then $[a_{i},a_{j}]\neq 0$, as should $[a_{i},a_{j}]=0$, then $\operatorname{span}_{\mathbb{F}}\{a_{i},a_{j}\}$ would be a 2-dimensional abelian subalgebra of $\g$ in direct contradiction to the assumption. Thus, by \eqref{qz}, $\kappa_{i}=\kappa_{j}$ for all $i\neq j$. So, we conclude that $\kappa(a_{i})=\alpha a_{i}$ for some $\alpha\in\mathbb{F}$, and then $\kappa=\alpha\, \mathrm{id}$.
\end{proof}
Theorem \ref{semisimplefibre} and Proposition \ref{nontriv_subal} significantly simplify the study of certain specific examples of Leibniz affgebras with Lie fibres, thereby facilitating their classification, as we shall see below.

In the following collection of examples we will focus on examples of 2-dimensional homogeneous Leibniz affgebras. Cuvier provided a complete classification of 2-dimensional Leibniz algebras in \cite[Lemme 2]{CUV}, where he proved that, up to isomorphisms, there are exactly four such algebras:
\begin{itemize}
    \item $L_{1}$, the abelian Leibniz algebra.
    \item $L_{2}$, the Borel subalgebra of the Lie algebra $sl_2(\mathbb{F})$, with table of multiplication given by
\[[e_{1},e_{1}]=0=[e_{2},e_{2}],\quad [e_{1},e_{2}]=-[e_{2},e_{1}]=e_{1}.\]
    \item The Leibniz algebra $L_{3}$, with table of multiplication given by
\[[e_{1},e_{1}]=[e_{2},e_{2}]=[e_{2},e_{1}]=0,\quad [e_{1},e_{2}]=e_{1}.\]
    \item The Leibniz algebra $L_{4}^{\xi}$, with table of multiplication given by
\[[e_{1},e_{1}]=[e_{1},e_{2}]=[e_{2},e_{1}]=0,\quad[e_{2},e_{2}]=\xi e_{1},\]
where $\xi\in\mathbb{F}/\mathbb{F}^{2}$.  
\end{itemize}
Clearly, $\xi\neq0$ since $L_{4}^{0} =L_1$. The parameter $\xi$ can be normalized to 1. Indeed, the linear map $\psi\colon L_{4}^{\xi}\rightarrow L_{4}^{1}$ given by the matrix $\begin{pmatrix}\frac{\delta^{2}}{\xi}&\gamma\\0&\delta\end{pmatrix}$, with $\delta\neq 0$, is a Leibniz algebra isomorphism. Consequently, $L_{4}^{\xi}$ and $L_{4}^{1}$ belong to the same isomorphism class. Therefore, without loss of generality, we will assume $\xi=1$  and write $L_4$ for $L_{4}^{1}$.

Any 2-dimensional homogeneous Leibniz affgebra must have (up to isomorphisms) one of these Leibniz algebras as its fibre. The abelian case has already been discussed in Example \ref{example_abelian_fibre}, so we will focus on the other cases.
\begin{example}[Homogeneous Leibniz affgebras fibred by $L_{2}$]\label{L2.2dim.ex}
    Note that $L_{2}$ is a Lie algebra which does not contain non-trivial abelian subalgebras. Thus, by Proposition \ref{nontriv_subal}, every homogeneous Leibniz affgebra fibred by $L_{2}$ is of the form $\mathfrak{a}(L_{2};\alpha\,\mathrm{id}-\mu,\mu,s)$ for $\alpha\in\mathbb{F}$, $s=s_{1}e_{1}+s_{2}e_{2},$ and $\mu=\begin{pmatrix}\mu_1&\mu_2\\\mu_3&\mu_4\end{pmatrix}.$ By Proposition \ref{imp_Leib}, $\mathfrak{a}(L_{2};\alpha\,\mathrm{id}-\mu,\mu,s)$ is a Leibniz affgebra if and only if conditions \eqref{Leib_cond_th} hold, which, in this particular case, reduces to the following two equalities:
    \begin{subequations}\label{L2conditions}
    \begin{gather}
        \label{L2cond1} [\mu(a),b]+[a,\mu(b)]=\mu([a,b]),\\
        \label{L2cond2} [s,a]=(\alpha-1)\mu(a)
    \end{gather}
    \end{subequations}
    for all $a,b\in L_{2}$. Given the definition of the Lie bracket on $L_{2}$, it is straightforward to check that conditions \eqref{L2conditions} hold if and only if 
    \[\mu=\begin{pmatrix}\mu_1&\mu_2\\0&0\end{pmatrix}, \quad s=\begin{pmatrix}(\alpha-1)\mu_{2}\\(1-\alpha)\mu_{1}\end{pmatrix},\]
    where the constraint $\mu_{3}=0=\mu_{4}$ follows from setting $a=e_{1}$ and $b=e_{2}$ in \eqref{L2cond1}, while the equalities $s_{1}=(\alpha-1)\mu_{2}$ and $s_{2}=(1-\alpha)\mu_{1}$ follow from taking $a=e_{1}$ and $a=e_{2}$ in \eqref{L2cond2}, respectively.

    One can easily show that an automorphism $\psi$ of $L_{2}$ has a matrix of the form $\psi=\begin{pmatrix}
        \theta & \delta\\0&1
    \end{pmatrix}$ with $\theta\neq 0$. Then, by applying Corollary \ref{iso_charac} with $q=q_{1}e_{1}+q_{2}e_{2}$, and taking into account that $\operatorname{ad}_{q}=\begin{pmatrix}
        q_{2}&-q_{1}\\0&0
    \end{pmatrix}$, it follows that $\mathfrak{a}(L_{2};\alpha\,\mathrm{id}-\mu,\mu,s)$ is isomorphic to $\mathfrak{a}(L_{2};\lambda',\mu',s')$ if and only if, 
    \begin{gather*}\lambda'=\begin{pmatrix}\alpha-(\mu_{1}+q_{2})&\delta(\mu_{1}+q_{2})+\theta(-\mu_{2}+q_{1})\\0&\alpha\end{pmatrix},\quad \mu'=\begin{pmatrix}\mu_{1}+q_{2}&-\delta(\mu_{1}+q_{2})+\theta(\mu_{2}-q_{1})\\0&0\end{pmatrix},\\ s'=\psi\begin{pmatrix}(1-\alpha)(-\mu_{2}+q_{1})\\(1-\alpha)(\mu_{1}+q_{2})\end{pmatrix}.\end{gather*}
    Hence, by setting $q_{1}=\mu_{2}$ and $q_{2}=-\mu_{1}$, we obtain that $\lambda'=\alpha\, \mathrm{id}$, $\mu'=0$ and $s'=0$, which leads to the conclusion that
    \[\mathfrak{a}(L_{2};\alpha\,\mathrm{id}-\mu,\mu,s)\simeq \mathfrak{a}(L_{2};\alpha\,\mathrm{id},0,0).\]
\end{example}

\begin{example}[Homogeneous Leibniz affgebras fibred by $L_{3}$]\label{L3.2dim.ex}
To determine all homogeneous Leibniz affgebras fibred by $L_{3}$, we begin by identifying the requirements that the endomorphisms $\lambda=\begin{pmatrix}
    \lambda_{1}&\lambda_{2}\\
    \lambda_{3}&\lambda_{4}
\end{pmatrix}$ and $\mu=\begin{pmatrix}\mu_1&\mu_2\\\mu_3&\mu_4\end{pmatrix}$ and the element $s=\begin{pmatrix}
    s_{1}\\s_{2}
\end{pmatrix}$ must satisfy in order to fulfill the conditions \eqref{Leib_cond_th}. It is straightforward to verify that $\mathfrak{a}(L_{3};\lambda,\mu,s)$ is a homogeneous Leibniz affgebra if and only if 
\[\lambda=\begin{pmatrix}
    \lambda_{1}&0\\0&\lambda_{4}
\end{pmatrix},\quad\mu=\begin{pmatrix}0&\mu_2\\0&0\end{pmatrix}, \quad s=\begin{pmatrix}\mu_{2}(\lambda_{1}-1)\\s_{2}\end{pmatrix},\]
where the resulting constraints arise from specific substitutions in conditions \eqref{Leib_cond_th}. Setting $a=e_{1}=b$ in \eqref{Leib_cond_th1} yields $\mu_{3}=0$; taking $a=e_{1}$ and $b=e_{2}$ in the same condition gives $\mu_{4}=0$ and $\lambda_{3}=0$; while choosing $a=e_{2}=b$ implies $\lambda_{2}=0$. From \eqref{Leib_cond_th2}, setting $a=e_{2}$ and $b=e_{1}$ leads to $\mu_{1}=0$; and finally, taking $a=e_{2}$ in \eqref{Leib_cond_th3} yields $s_{1}=\mu_{2}(\lambda_{1}-1)$. 

It is easy to see that an automorphism of $L_{3}$ is of the form $\psi=\begin{pmatrix}
    \theta&0\\0&1
\end{pmatrix}$ with $\theta\neq 0$. Hence, by applying Corollary \ref{iso_charac} with $q=q_{1}e_{1}+q_{2}e_{2}$ and noting that $\operatorname{ad}_{q}=\begin{pmatrix}
    q_{2}&0\\0&0
\end{pmatrix}$ and $[q,-]=\begin{pmatrix}
    0&q_{1}\\0&0
\end{pmatrix}$, we obtain that $\mathfrak{a}(L_{3};\lambda,\mu,s)$ is isomorphic to $\mathfrak{a}(L_{3};\lambda',\mu',s')$ if and only if
\[\lambda'=\begin{pmatrix}
    \lambda_{1}-q_{2}&0\\0&\lambda_{4}
\end{pmatrix},\quad \mu'=\begin{pmatrix}
    0&\theta(\mu_{2}-q_{1})\\0&0
\end{pmatrix},\quad s'=\psi\begin{pmatrix}
    (\mu_{2}-q_{1})(-1+\lambda_{1}-q_{2})\\s_{2}+q_{2}(1-\lambda_{4})
\end{pmatrix}.\]
As a consequence, if we set $q_{1}=\mu_{2}$ and $q_{2}=\lambda_{1}$, then 
\[
\mathfrak{a}(L_{3};\lambda,\mu,s)\simeq \mathfrak{a}(L_{3};\lambda',0,s_{2}'e_{2}),
\]
where $\lambda'=\begin{pmatrix}
 0&0\\0&\lambda_{4}
\end{pmatrix}$ and $s_{2}'=s_{2}+\lambda_{1}(1-\lambda_{4})$.
\end{example}

\begin{example}[Homogeneous Leibniz affgebras fibred by $L_{4}$]\label{L4.2dim.ex}
The process we will follow to study the homogeneous Leibniz affgebras fibred by $L_{4}$ is analogous to that carried out in Example \ref{L3.2dim.ex}. To begin with, $\mathfrak{a}(L_{4};\lambda,\mu,s)$ is a homogeneous Leibniz affgebra, i.e., $\lambda=\begin{pmatrix}\lambda_1&\lambda_2\\\lambda_3&\lambda_4\end{pmatrix}$, $\mu=\begin{pmatrix}\mu_1&\mu_2\\\mu_3&\mu_4\end{pmatrix}$ and $s=\begin{pmatrix}
    s_{1}\\s_{2}
\end{pmatrix}$ satisfy conditions \eqref{Leib_cond_th}, if and only if 
\[\lambda=\begin{pmatrix}\lambda_1&\lambda_2\\0&\lambda_1\end{pmatrix},\quad \mu=\begin{pmatrix}0&\mu_2\\0&0\end{pmatrix},\quad s=\begin{pmatrix}
    s_{1}\\\mu_{2}(\lambda_{1}-1)
\end{pmatrix},\]
where the constraints obtained follow from specific instances of conditions \eqref{Leib_cond_th}. Setting  $a = e_2 = b$ in \eqref{Leib_cond_th1} yields $\lambda_3 = 0$ and $\mu_4 = \lambda_4 - \lambda_1$, while taking  $a = e_2$, $b = e_1$ in the same condition gives $\mu_3 = 0$. From \eqref{Leib_cond_th2}, choosing $a = e_2 = b$ leads to $\mu_1 = 0$, and finally, substituting $a = e_2$ into \eqref{Leib_cond_th3} it is obtained that $s_{2}=\mu_{2}(\lambda_{1}-1)$ and $\lambda_{1}=\lambda_{4}$. As a consequence, $\mu_{4}=0$.

Now, one easily obtain that automorphisms of $L_{4}$ are of the form $\psi=\begin{pmatrix}
    \theta^{2}&\delta\\0&\theta
\end{pmatrix}$ with $\theta\neq 0$, and then, by applying Corollary \ref{iso_charac}, $\mathfrak{a}(L_{4};\lambda,\mu,s)$ is isomorphic to $\mathfrak{a}(L_{4};\lambda',\mu',s')$ if and only if 
\[\lambda'=\begin{pmatrix}
    \lambda_{1} & \theta(\lambda_{2}- q_{2})\\0&\lambda_{1}
\end{pmatrix},\quad\mu'=\begin{pmatrix}
    0 & \theta(\mu_{2}- q_{2})\\0&0
\end{pmatrix}.\]
By choosing $q_{2}=\mu_{2}$, it implies that $\mu'=0$ and $\lambda'=\begin{pmatrix}
    \lambda_{1} & \theta(\lambda_{2}-\mu_{2})\\0&\lambda_{1}
\end{pmatrix}$. In case $\mu_{2}=\lambda_{2}$, $\lambda'=\lambda_{1}\,\mathrm{id}$ and it is easy to show that $s'=\theta^{2}\left(s_{1}-\lambda_{2}^{2}\right)e_{1}$ by setting $q_{1}=0$. Hence, $s'=\Lambda e_{1}$, where $\Lambda$ is any element in the square class of $s_{1}-\lambda_{2}^{2}$.
When $\mu_{2}\neq \lambda_{2}$, by setting $\theta=\frac{1}{\lambda_{2}-\mu_{2}}$, $\lambda'=\begin{pmatrix}
    \lambda_{1}&1\\0&\lambda_{1}
\end{pmatrix}$ and $s'=\left(\frac{ s_{1}-\lambda_{2}\mu_{2}}{(\mu_{2}-\lambda_{2})^{2}}\right)e_{1}$ for $q_{1}=0$. 
Then, we conclude the following:
\[\mathfrak{a}(L_{4};\lambda,\mu,s)\simeq\begin{cases}
\mathfrak{a}\left(L_{4};\lambda_{1}\,\mathrm{id},0,\Lambda e_{1}\right),&\lambda_{2}=\mu_{2},\\
    \mathfrak{a}(L_{4};\lambda',0,\Omega\,e_{1}),& \lambda_{2}\neq\mu_{2},
\end{cases}\]
where $\Lambda$ is any element in the square class of $s_{1}-\lambda_{2}^{2}$ or 0, $\lambda'=\begin{pmatrix}
    \lambda_{1}&1\\0&\lambda_{1}
\end{pmatrix}$ and $\Omega=\frac{ s_{1}-\lambda_{2}\mu_{2}}{(\mu_{2}-\lambda_{2})^{2}}$.
\end{example}

We give now two examples of 3-dimensional homogeneous Leibniz affgebras.
\begin{example}[Homogeneous Leibniz affgebra fibred by $\mathfrak{sl}(2,\mathbb{C})$]
    For this example we will consider that $\mathbb{F}=\mathbb{C}$ the field of complex numbers. It is well-known that $\mathfrak{sl}(2,\mathbb{C})$ is a complex simple Lie algebra, and then, any homogeneous Leibniz affgebra with fibre $\mathfrak{sl}(2,\mathbb{C})$, $\mathfrak{a}(\mathfrak{sl}(2,\mathbb{C});\lambda,\mu,s)$, is automatically isomorphic to $\mathfrak{a}(\mathfrak{sl}(2,\mathbb{C});\alpha\, \mathrm{id},0,0)$ by Theorem \ref{semisimplefibre}.
\end{example}

\begin{example}[Homogeneous 3-dimensional Leibniz affgebra fibred by a non-Lie, Leibniz algebra]
    In the present example again we consider $\mathbb{F}=\mathbb{C}$. We aim to determine the conditions that the maps $\lambda$, $\mu$ and the vector $s$ must satisfy in order for $\mathfrak{a}(L_{7};\lambda,\mu,s)$ to define a homogeneous Leibniz affgebra, where $L_{7}$ denotes the non-Lie Leibniz algebra whose multiplication table is given by:
    \[[e_{1},e_{2}]=[e_{1},e_{3}]=[e_{3},e_{2}]=[e_{3},e_{3}]=e_{1};\]
    see \cite[Table 3.1]{AOR}. In other words, if $\lambda,$ $\mu$ and $s$ are denoted by
\begin{gather*}
    \lambda=\begin{pmatrix}
        \lambda_{1}&\lambda_{2}&\lambda_{3}\\\lambda_{4}&\lambda_{5}&\lambda_{6}\\\lambda_{7}&\lambda_{8}&\lambda_{9}
    \end{pmatrix},\quad\mu=\begin{pmatrix}
        \mu_{1}&\mu_{2}&\mu_{3}\\\mu_{4}&\mu_{5}&\mu_{6}\\\mu_{7}&\mu_{8}&\mu_{9}
    \end{pmatrix},\quad s=\begin{pmatrix}
        s_{1}\\s_{2}\\s_{3}
    \end{pmatrix},
\end{gather*}
then we are thus interested in determining under which conditions the identities \eqref{Leib_cond_th} hold. 

First of all let us assume that \eqref{Leib_cond_th} hold. The following constraints are obtained by evaluating the conditions \eqref{Leib_cond_th1}, \eqref{Leib_cond_th2} and \eqref{Leib_cond_th3} on the canonical basis $\{e_1, e_2, e_3\}$ of $L_7$.
From \eqref{Leib_cond_th1}, we obtain:
\begin{itemize}
    \item[(i)] For $a = e_1 = b$, it follows that $\mu_7 = -\mu_4$.
    \item[(ii)] For $a = e_2 = b$, we deduce $\lambda_8 = -\lambda_2$.
    \item[(iii)] Taking $a = e_3 = b$ yields $\lambda_{4}=0=\lambda_{7}$ and $\lambda_1 = \lambda_3 + \lambda_9 - \mu_6 - \mu_9$.
    \item[(iv)] Setting $a = e_1$, $b = e_2$ gives $\mu_8 = -\mu_5$.
    \item[(v)] With $a = e_1$, $b = e_3$, we get $\mu_9 = -\mu_6$. Substituting into (iii), we obtain the simplified relation $\lambda_1 = \lambda_3 + \lambda_9$.
\end{itemize}
From \eqref{Leib_cond_th2} we derive:
\begin{itemize}
    \item[(vi)] For $a = e_3 = b$, we have $\mu_1 = 0=\mu_{4}$.
    \item[(vii)] For $a = e_2$, $b = e_3$, the identity \begin{gather}\label{c01}\mu_2 - \mu_5 = \mu_3 - \mu_6\end{gather} must hold.
\end{itemize}
Finally, from \eqref{Leib_cond_th3}, we obtain:
\begin{itemize}
    \item[(viii)] For $a = e_2$:
    \begin{subequations}\label{c02}\begin{gather}
        s_1 + s_3 = \mu_2(\lambda_3 + \lambda_9 - 1) + \mu_5(\lambda_2 + \mu_2 - \lambda_3 - \mu_3), \\
        \mu_5(\lambda_5 - \lambda_6 - 1 + \mu_5 - \mu_6) = 0, \\
        \mu_5(\lambda_2 + \lambda_9 - 1 + \mu_5 - \mu_6) = 0.
    \end{gather}\end{subequations}
    
    \item[(ix)] For $a = e_3$:
    \begin{subequations}\label{c03}\begin{gather}
        s_1 + s_3 = \mu_3(\lambda_3 + \lambda_9 - 1) + \mu_6(\lambda_2 + \mu_2 - \lambda_3 - \mu_3), \\
        \mu_6(\lambda_5 - \lambda_6 - 1 + \mu_5 - \mu_6) = 0, \\
        \mu_6(\lambda_2 + \lambda_9 - 1 + \mu_5 - \mu_6) = 0.
    \end{gather}
    \end{subequations}
\end{itemize}
The conditions obtained above are only necessary; however, it is straightforward to verify that, in this case, they are also sufficient. In conclusion, $\mathfrak{a}(L_{7};\lambda,\mu,s)$ is a homogeneous Leibniz affgebra if and only if
\begin{gather*}
    \lambda=\begin{pmatrix}
        \lambda_{3}+\lambda_{9}&\lambda_{2}&\lambda_{3}\\0&\lambda_{5}&\lambda_{6}\\0&-\lambda_{2}&\lambda_{9}
    \end{pmatrix},\quad\mu=\begin{pmatrix}
        0&\mu_{2}&\mu_{3}\\0&\mu_{5}&\mu_{6}\\0&-\mu_{5}&-\mu_{6}
    \end{pmatrix},
\end{gather*}
subject to \eqref{c01}, \eqref{c02} and \eqref{c03}.

Assuming certain conditions on the parameters, five different families of Leibniz affgebras with fibre $L_{7}$ are obtained:
\begin{gather}\tag{F1}\label{f1l7}
    \lambda=\begin{pmatrix}
        \lambda_{3}+\lambda_{9}&\lambda_{2}&\lambda_{3}\\0&\lambda_{5}&\lambda_{6}\\0&-\lambda_{2}&\lambda_{9}
    \end{pmatrix},\quad\mu=\begin{pmatrix}
        0&\mu_{2}&\mu_{2}\\0&0&0\\0&0&0
    \end{pmatrix}\quad \text{ if } \mu_{5}=0=\mu_{6}
\end{gather}
subject to $$s_{1}+s_{3}=\mu_{2}(\lambda_{3}+\lambda_{9}-1),$$
\begin{gather}\tag{F2}\label{f2l7}
    \lambda=\begin{pmatrix}
        \lambda_{3}+\lambda_{9}&\lambda_{2}&\lambda_{3}\\0&\lambda_{5}&\lambda_{6}\\0&-\lambda_{2}&\lambda_{9}
    \end{pmatrix},\quad\mu=\begin{pmatrix}
        0&\mu_{2}&\mu_{3}\\0&0&\lambda_{2}+\lambda_{9}-1\\0&0&1-\lambda_{2}-\lambda_{9}
    \end{pmatrix}\quad\text{ if } \mu_{5}=0,\,\mu_{6}\neq0
\end{gather}
subject to
\begin{gather*}
    \mu_{3}-\mu_{2}=\lambda_{2}+\lambda_{9}-1,\qquad
    s_{1}+s_{3}=\mu_{2}(\lambda_{3}+\lambda_{9}-1),\qquad
    \lambda_{2}-\lambda_{5}+\lambda_{6}+\lambda_{9}=0,
\end{gather*}
\begin{gather}\tag{F3}\label{f3l7}
    \lambda=\begin{pmatrix}
        \lambda_{3}+\lambda_{9}&\lambda_{2}&\lambda_{3}\\0&\lambda_{5}&\lambda_{6}\\0&-\lambda_{2}&\lambda_{9}
    \end{pmatrix},\quad\mu=\begin{pmatrix}
        0&\mu_{2}&\mu_{3}\\0&1-\lambda_{2}-\lambda_{9}&0\\0&\lambda_{2}+\lambda_{9}-1&0
    \end{pmatrix}\quad \text{ if }\mu_{5}\neq0,\,\mu_{6}=0
\end{gather}
subject to
\begin{gather*}
    \mu_{2}-\mu_{3}=1-\lambda_{2}-\lambda_{9},\qquad
    s_{1}+s_{3}=\mu_{3}(\lambda_{3}+\lambda_{9}-1),\qquad
    \lambda_{2}-\lambda_{5}+\lambda_{6}+\lambda_{9}=0,
\end{gather*}
\begin{gather}\tag{F4}\label{f4l7}
    \lambda=\begin{pmatrix}
        \lambda_{3}-\lambda_{2}+1&\lambda_{2}&\lambda_{3}\\0&\lambda_{5}&\lambda_{5}-1\\0&-\lambda_{2}&1-\lambda_{2}
    \end{pmatrix},\quad\mu=\begin{pmatrix}
        0&\mu_{2}&\mu_{2}\\0&\mu_{5}&\mu_{5}\\0&-\mu_{5}&-\mu_{5}
    \end{pmatrix}\quad \text{ if }\begin{array}{cc}
       \mu_{5},\,\mu_{6}\neq 0 \\
     \mu_{2}=\mu_{3}
    \end{array}
\end{gather}
subject to $$s_{1}+s_{3}=(\lambda_{3}-\lambda_{2})(\mu_{2}-\mu_{5}),$$ and
\begin{gather}\tag{F5}\label{f5l7}
    \lambda=\begin{pmatrix}
        \lambda_{3}+\lambda_{5}-\lambda_{2}-\lambda_{6}&\lambda_{2}&\lambda_{3}\\0&\lambda_{5}&\lambda_{6}\\0&-\lambda_{2}&\lambda_{5}-\lambda_{2}-\lambda_{6}
    \end{pmatrix},\quad\mu=\begin{pmatrix}
        0&\mu_{2}&\mu_{3}\\0&\mu_{5}&\mu_{6}\\0&-\mu_{5}&-\mu_{6}
    \end{pmatrix}\quad\text{ if }\begin{array}{cc}
         \mu_{5},\,\mu_{6}\neq 0\\
         \mu_{2}\neq\mu_{3}
    \end{array}
\end{gather}
subject to
\begin{gather*}
    \mu_{2}-\mu_{3}=\mu_{5}-\mu_{6}=1-\lambda_{5}+\lambda_{6},\qquad
    s_{1}+s_{3}=(\mu_{2}-\mu_{5})(\lambda_{3}+\lambda_{5}-\lambda_{2}-\lambda_{6}-1).
\end{gather*}
\end{example}

Note that, 
\eqref{Leib_cond_th1} is equivalent to  \eqref{gen.deriv}, 
as in the case of Lie affgebras. In the latter case this is the only condition that needs to be imposed on $\lambda$, $\mu$ and $s$. This implies that not every Lie affgebra is a homogeneous Leibniz affgebra. For instance, in the situation of Example~\ref{example_abelian_fibre}, any choice of $\lambda$ gives an affine Lie bracket, but not necessarily a homogeneous Leibniz bracket. As the conditions for  derivative Leibniz brackets are even more restrictive, not every affine Lie bracket is a derivative Leibniz bracket either. The proposal for a class of Leibniz affgebras which includes all Lie affgebras is given in the subsequent Section~\ref{sec.Lie-type}.

\subsection{Lie-type Leibniz affgebras}\label{sec.Lie-type}
It is true that every Leibniz algebra with an antisymmetric bracket is a Lie algebra. In the affine case this depends on the form of the Leibnizian. This is more precisely stated in Proposition \ref{prop.Leib.Lie}.

\begin{lemma}\label{lem.Leib.Lie}
 Let $\{-,-\}$ be a bi-affine operation of $\mathfrak{a}$. Then, for all $o\in \mathfrak{a}$  the tri-affine map $\Lambda\colon \mathfrak{a}^3\to \mathfrak{a}$ given by 
 $$\Lambda(a,b,c)= \{a,a\}-\{\{c,c\},c\}+\{\{b,c\},\{b,c\}\}-\{\{a,a\},a\}+\{\{c,c\},b\}-\{\{b,b\},b\}+\{\{a,a\},b\}$$ satisfies \eqref{cond_gen}.  
\end{lemma}

\begin{definition}\label{Lie-type_LeibnizAffgebra}
A Leibniz affgebra $(\mathfrak{a},\{-,-\})$ with  Leibnizian $\Lambda(a,b,c)$ as in Lemma \ref{lem.Leib.Lie} is called 
a \textbf{Lie-type Leibniz affgebra}.

\end{definition}

The terminology introduced in Definition~\ref{Lie-type_LeibnizAffgebra} is justified in the following proposition.
\begin{proposition}\label{prop.Leib.Lie}
    Let $(\mathfrak{a},\{-,-\})$ be  a  Leibniz affgebra with bracket satisfying the affine antisymmetry \eqref{aff_antisy}. Then, $(\mathfrak{a},\{-,-\})$ is a  Lie affgebra if and only if $(\mathfrak{a},\{-,-\})$ is a Lie-type Leibniz affgebra.
\end{proposition}
\begin{proof}
    If $(\mathfrak{a},\{-,-\})$ is a Leibniz affgebra whose affine bracket $\{-,-\}$ satisfies the affine antisymmetry \eqref{aff_antisy}, then the Leibnizian takes the following form: 
    \begin{align}\label{demLeibLieType}
        \Lambda(a,b,c)&=\{\{a,c\},b\}-\{\{a,b\},c\}+\{a,\{b,c\}\}\;\footnotesize\textnormal{(by \eqref{Leibnizian})}\\\nonumber
        &=\{\{a,a\}-\{c,a\}+\{c,c\},b\}-\{\{a,b\},c\}
        \\\nonumber&+\{a,a\}-\{\{b,c\},a\}+\{\{b,c\},\{b,c\}\}\;\footnotesize\textnormal{(by \eqref{aff_antisy})}\\\nonumber
        &=\{a,a\}+\{\{b,c\},\{b,c\}\}+\{\{a,a\},b\}+\{\{c,c\},b\}\\\nonumber
        &-\{\{c,a\},b\}-\{\{a,b\},c\}-\{\{b,c\},a\}.
    \end{align}
Hence, expression \eqref{demLeibLieType} for the Leibnizian coincides with the one given in Lemma \ref{lem.Leib.Lie} if and only if the affine Jacobi identity \eqref{aff_jac} holds. In other words, $(\mathfrak{a},\{-,-\})$ is a Lie-type Leibniz affgebra if and only if $(\mathfrak{a},\{-,-\})$ is a Lie affgebra.
\end{proof} 

Construction of Lie-type Leibniz affgebras with a prescribed Leibniz algebra fibre $(\mathfrak{l},[-,-])$ via the formula \eqref{th_affLeib_bracket} boils down to solving the following seven equations that result from equations \eqref{gen.Leibnizian} in Theorem~\ref{mainth_phi_gen}:
\begin{subequations}\label{gen.Leibnizian.Lie}
    \begin{gather}
    \label{Leib_Lie_cond_0}
    [s,s]=0\\
    \label{Leib_Lie_cond_1}
    [a,a]-[[a,a],a]-[\lambda(a),a]-[\mu(a),a]-[s,a]-[a,s]=0\\
    \label{Leib_Lie_cond_2}
    [\lambda(b),\lambda(b)]+[\lambda(b),s]+[s,\lambda(b)]-[[b,b],b]-[\lambda(b),b]-[\mu(b),b]-\lambda([b,b])=0\\
    \label{Leib_Lie_cond_3}
    -[[c,c],c]-[\lambda(c),c]-[\mu(c),c]+[\mu(c),\mu(c)]+[\mu(c),s]+[s,\mu(c)]=0\\
    \label{Leib_Lie_cond_4}
    [[a,a],b]+[\mu(a),b]-[a,\lambda(b)]+\lambda([a,b])=0\\
    \label{Leib_Lie_cond_5}
    \lambda([a,c])+[a,\mu(c)]-[\lambda(a),c]=0\\\nonumber
    [[b,c],[b,c]]+[\lambda(b),[b,c]]+[\mu(c),[b,c]]+[s,[b,c]]+[[b,c],\lambda(b)]+[[b,c],\mu(c)]+  \\\label{Leib_Lie_cond_6}
    [[b,c],s]+[\lambda(b),\mu(c)]+[\mu(c),\lambda(b)]+\lambda([b,c])+[[c,c],b]+[\lambda(c),b]+[\mu(b),c]=0 
    \end{gather}
\end{subequations}
for all $a,b,c\in\mathfrak{l}$.
\begin{example}
In the case of the Leibniz algebra $L_3$, conditions \eqref{gen.Leibnizian.Lie} imply that all $\lambda$, $\mu$ and $s$ must be zero. 
This shows that not all Leibniz algebras truly have affine versions of Lie-type.
\end{example}
\begin{example}
    Consider the 2-dimensional Leibniz algebra $L_{4}$. To determine all Lie-type Leibniz affgebras with fibre $L_{4}$ we need to check conditions \eqref{gen.Leibnizian.Lie} according to Theorem \ref{mainth_phi_gen}. Let us assume the following notation: $\lambda=\begin{pmatrix}\lambda_1&\lambda_2\\\lambda_3&\lambda_4\end{pmatrix},$ $ \mu=\begin{pmatrix}\mu_1&\mu_2\\\mu_3&\mu_4\end{pmatrix},$ $s=\begin{pmatrix}s_{1}\\s_2\end{pmatrix}.$ The resulting constraints are:
    \begin{itemize}
    \item[(i)] From \eqref{Leib_Lie_cond_0}, we obtain immediately that $s_2=0$.
    \item[(ii)] From \eqref{Leib_Lie_cond_1}, by taking $a=e_2$, it yields the constraint
    $1-\lambda_4-\mu_4=0.$
    \item[(iii)] By setting $b=e_{2}$ in \eqref{Leib_Lie_cond_2} we obtain that 
    $\lambda_{3}=0$\text{ and }
        $-\lambda_4-\mu_4-\lambda_1+\lambda_4^2=0.$
    \item[(iv)] Similarly, setting $c=e_{2}$ in \eqref{Leib_Lie_cond_3}, it yields to 
    $-\lambda_4-\mu_4+\mu_4^2=0.$
    \item[(v)] Furthermore, from condition \eqref{Leib_Lie_cond_4}, it is obtained that $\mu_{3}=0$ when we set $a=e_{1}$ and $b=e_{2}$, and also
    $\mu_{4}-\lambda_{4}+\lambda_{1}=0$
    by taking $a=e_{2}=b.$
    \item[(vi)] From equation \eqref{Leib_Lie_cond_6}, setting $b=e_{2}=c$, it is obtained that 
    $2\lambda_{4}\mu_{4}+\lambda_{1}+\lambda_{4}+\mu_{4}=0.$
    \end{itemize}
Solving the resulting system of equations, we obtain two parametric families of Leibniz affgebras:
\begin{enumerate}
   \item[(F1)]$ \qquad \lambda=\begin{pmatrix}-1&\lambda_2\\0&0\end{pmatrix} \qquad \mu=\begin{pmatrix}\mu_1&\mu_2\\0&1\end{pmatrix} \qquad \; \:s=\begin{pmatrix}s_{1}\\0\end{pmatrix}$,
    
    \item[(F2)] $ \qquad \lambda=\begin{pmatrix}3&\lambda_2\\0&2\end{pmatrix} \quad \; \; \; \quad \mu=\begin{pmatrix}\mu_1&\mu_2\\0&-1\end{pmatrix} \qquad s=\begin{pmatrix}s_{1}\\0\end{pmatrix}$.
    \end{enumerate}
    These families were derived by evaluating the compatibility conditions \eqref{gen.Leibnizian.Lie} on the basis elements $e_{1}$ and $e_{2}$. To guarantee sufficiency, we must also ensure that these constraints hold for arbitrary linear combinations $\alpha_{1}e_{1}+\alpha_{2}e_{2}$, but no extra conditions arise from this check.

    As we have already seen in Example \ref{L4.2dim.ex}, automorphisms of $L_4$ have the form $\psi=\begin{pmatrix}
    \theta^{2}&\delta\\0&\theta
\end{pmatrix}$ with $\theta\neq 0$. By applying Corollary \ref{iso_charac} for the family (F1), $\mathfrak{a}(L_{4};\lambda,\mu,s)$ is isomorphic to $\mathfrak{a}(L_{4};\lambda',\mu',s')$ if and only if 
\[\lambda'=\begin{pmatrix}
    -1 & \frac{\delta+\theta^2\lambda_2-\theta^2 q_2}{\theta}\\0&0
\end{pmatrix},\quad\mu'=\begin{pmatrix}
    \mu_1 & \frac{\delta-\theta^2 q_2-\delta\mu_1+\theta^2\mu_2}{\theta}\\0&1
\end{pmatrix}.\]
By choosing $q_2=\frac{\delta+\theta^2\lambda_2}{\theta^2}$, it implies that $\lambda'=\begin{pmatrix}
    -1&0\\0&0
\end{pmatrix}$ and $\mu'=\begin{pmatrix}
    \mu_1 &\frac{\theta^2(\mu_2-\lambda_2)-\delta\mu_1}{\theta}\\0 &1
\end{pmatrix}$. In case $\mu_2=\lambda_2$, choosing $\delta=0$, $\mu'=\begin{pmatrix}
    \mu_1 &0\\0&1
\end{pmatrix}$ and it implies that $s'=\begin{pmatrix}
    \theta^2((2-\mu_{1})q_1+s_1-\lambda_2^2)\\
    0
\end{pmatrix}$. Consequently, for $\mu_{1}$ we obtain two different cases: if $\mu_1\neq 2$, by choosing $q_1=\frac{\lambda_2^2-s_1}{2-\mu_1}$, then we obtain $s'=0$; and if $\mu_1=2$, then $s'=\begin{pmatrix}
    \theta^2(s_1-\lambda_2^2)\\0
\end{pmatrix}$, which can be identified within each square class of $s_1-\lambda_2^2$. When $\mu_2\neq \lambda_2$ and $\mu_1=0$, by setting $\theta=\frac{1}{\mu_2-\lambda_2}$, $\mu'=\begin{pmatrix}
    0&1\\0&1
\end{pmatrix}$ and $s'=0$ by setting $q_1=\frac{\lambda_2\mu_2-s_1}{2}$. The last case is when $\mu_2\neq\lambda_2$ and $\mu_1\neq 0$. By choosing $\delta=\frac{\theta^2(\mu_2-\lambda_2)}{\mu_1}$, it implies that $\mu'=\begin{pmatrix}
    \mu_1 &0\\
    0 &1
\end{pmatrix}$ and $s'=\theta^2\left(2q_1+\frac{(\mu_2+\lambda_2(\mu_1-1))^2}{\mu_1^2}-\frac{\mu_2\lambda_2+\lambda_2^2(\mu_1-1)}{\mu_1}+s_1-q_1\mu_1-\frac{\mu_2^2+\lambda_2\mu_2(\mu_1-1)}{\mu_1}\right)e_1$. If $\mu_1\neq 2$, choosing appropriate $q_1$, we can set $s'=0$. In case when $\mu_1=2$ we obtain $s'=
    \theta^2\left(s_1-\left(\frac{(\lambda_2+\mu_2)}{2}\right)^2\right)e_1$.

In short, for the family (F1) we have obtained the following isomorphism classes:
\[\mathfrak{a}(L_{4};\lambda,\mu,s)\simeq \begin{cases}
\mathfrak{a}\left(L_{4};\lambda'=\begin{pmatrix}-1&0\\0&0\end{pmatrix},\mu'=\begin{pmatrix}2&0\\0&1\end{pmatrix},s'=\begin{pmatrix}\Omega\\0\end{pmatrix}\right),\quad \;\mu_1= 2\\
\mathfrak{a}\left(L_{4};\lambda'=\begin{pmatrix}-1&0\\0&0\end{pmatrix},\mu'=\begin{pmatrix}\mu_1&0\\0&1\end{pmatrix},  s'=\begin{pmatrix}0\\0\end{pmatrix}\right),   \quad  \;\mu_1\neq 0,2 \; \textnormal{or}\; \begin{array}{cc}\lambda_2=\mu_2\\\mu_1=0\end{array}\\
\mathfrak{a}\left(L_{4};\lambda'=\begin{pmatrix}-1&0\\0&0\end{pmatrix},\mu'=\begin{pmatrix}0&1\\0&1\end{pmatrix},s'=\begin{pmatrix}0\\0\end{pmatrix}\right),\quad \lambda_2\neq \mu_2, \;\mu_1=0,
\end{cases}\]
 where $\Omega$ is any element of the square class of $s_1-\left(\frac{\lambda_2+\mu_2. }{2}\right)^2$ or $0$. 

 By applying Corollary \ref{iso_charac} for the family (F2), $\mathfrak{a}(L_{4};\lambda,\mu,s)$ is isomorphic to $\mathfrak{a}(L_{4};\lambda',\mu',s')$ if and only if 
\[\lambda'=\begin{pmatrix}
    3 & \frac{-\delta+\theta^2\lambda_2-\theta^2 q_2}{\theta}\\0&2
\end{pmatrix},\quad\mu'=\begin{pmatrix}
    \mu_1 & \frac{-\delta-\theta^2 q_2-\delta\mu_1+\theta^2\mu_2}{\theta}\\0&-1
\end{pmatrix}.\]
By means of analogous reasoning to those carried out for the family (F1), we obtain the following isomorphism classes for the family (F2):
\[\mathfrak{a}(L_{4};\lambda,\mu,s)\simeq \begin{cases}
\mathfrak{a}\left(L_{4};\lambda'=\begin{pmatrix}3&0\\0&2\end{pmatrix},\mu'=\begin{pmatrix}-2&0\\0&-1\end{pmatrix},s'=\begin{pmatrix}\Omega\\0\end{pmatrix}\right),\quad \;\mu_1= -2\\
\mathfrak{a}\left(L_{4};\lambda'=\begin{pmatrix}3&0\\0&2\end{pmatrix},\mu'=\begin{pmatrix}\mu_1&0\\0&-1\end{pmatrix},  s'=\begin{pmatrix}
    0\\0
\end{pmatrix}\right), \quad\mu_1\neq 0,-2 \; \textnormal{or}\; \begin{array}{cc}\lambda_2=\mu_2\\\mu_1=0\end{array}\\
\mathfrak{a}\left(L_{4};\lambda'=\begin{pmatrix}3&0\\0&2\end{pmatrix},\mu'=\begin{pmatrix}0&1\\0&-1\end{pmatrix},s'=\begin{pmatrix}0\\0\end{pmatrix}\right),\quad \lambda_2\neq \mu_2, \;\mu_1=0,
\end{cases}\]
 where $\Omega$ is any element of the square class of $s_1-\left(\frac{\lambda_2+\mu_2}{2}\right)^2$ or $0$. 
 \end{example}
 
 \section*{Funding}
The research  of Tomasz Brzezi\'nski and Krzysztof Radziszewski is partially supported by the National Science Centre, Poland, through the WEAVE-UNISONO (Grant no.\ 2023/05/Y/ST1/00046). 

The research of Brais Ramos Pérez is partially supported by Ministerio de Ciencia e Innovación, Spain (Grant no. PID2020-115155GB-I00) and Xunta de Galicia (Grant no. ED431C 2023/31 and ED481A-2023-023).

\bibliographystyle{amsalpha}

\end{document}